\documentclass[a4paper,10pt]{amsart}

\usepackage[latin1]{inputenc}
\usepackage[T1]{fontenc}

\usepackage{amsmath, amsthm, amssymb}
\usepackage{amscd}

\usepackage[super]{nth}

\usepackage[all]{xy}
\usepackage{dsfont}
\usepackage{mathtools}

\usepackage[usenames,dvipsnames]{pstricks}
\usepackage{epsfig}
\usepackage{pst-grad} 
\usepackage{pst-plot} 
\usepackage{pst-solides3d}
\usepackage{caption}
\usepackage{subcaption}
\usepackage{pstricks-add}

\usepackage[margin=2.5cm]{geometry}

\newtheorem{thm}{Theorem}[section]
\newtheorem{cor}[thm]{Corollary}
\newtheorem{prop}[thm]{Proposition}
\newtheorem{defin}[thm]{Definition}
\newtheorem{lem}[thm]{Lemma}
\newtheorem{claim}[thm]{Claim}
\newtheorem*{thm*}{Theorem}
\newtheorem*{prop*}{Proposition}
\newtheorem*{defin*}{Definition}
\newtheorem*{lem*}{Lemma}
\newtheorem*{claim*}{Claim}

\newtheorem{thmintro}{Theorem}

\theoremstyle{remark}
\newtheorem{rem}[thm]{Remark}
\newtheorem*{rem*}{Remark}

\newtheorem*{question*}{Question}
\newtheorem*{conjecture*}{Conjecture}
\usepackage{color}

\usepackage[colorlinks,citecolor=Blue]{hyperref}

\newcommand\R{\mathbb R}

\newcommand\N{\mathbb N}

\newcommand{\RP}{\mathbb{RP}}

\def\S{\mathbb S}

\newcommand\ada{A \wedge dA^{n-1}}

\def\L{\mathcal{L}}

\DeclareMathOperator{\Ker}{Ker}
\DeclareMathOperator{\voleucl}{\mathrm{vol}_{\mathrm{Eucl}}}
\DeclareMathOperator{\vol}{\mathrm{vol}}

\DeclareMathOperator{\id}{Id}

\newcommand\eps{\varepsilon}

\newcommand{\C}{\mathcal{C}}

\newcommand{\dC}{\partial \mathcal{C}}
\renewcommand{\d}{d_{\mathcal{C}}}
\newcommand{\F}{F_{\mathcal{C}}}
\newcommand{\Homeo}{\mathrm{Homeo}}
\newcommand{\E}{{\mathcal{E}}}
\newcommand{\n}{{\mathcal{N}}}
\renewcommand{\l}{\lambda}

\newcommand{\h}{\mathtt{h}} 

\newcommand{\cheegerapoids}{h_{\mathrm{Cheeger}}^{\sigma,\Omega}}
\newcommand{\cheeger}{h_{\mathrm{Cheeger}}}

\title{Laplacian and spectral gap in regular Hilbert geometries}

\author{Thomas Barthelm\'e}   
\address{Department of Mathematics, Tufts University, Medford, MA}
\email{thomas.barthelme@tufts.edu}
\urladdr{\href{https://sites.google.com/site/thomasbarthelme/}{https://sites.google.com/site/thomasbarthelme/}}

\author{Bruno Colbois}
\address{Institut de math\'ematiques, Universit\'e de Neuch\^{a}tel, Neuch\^{a}tel}
\email{bruno.colbois@unine.ch}

\author{Micka\"el Crampon}
\address{Departamento de Matem\'atica y Ciencia de la Computaci\'on, Universidad de Santiago de Chile, Santiago de Chile}
\email{mickael.crampon@usach.cl}
\urladdr{\href{http://mikl.crampon.free.fr/}{http://mikl.crampon.free.fr/}}

 \author{Patrick Verovic}
\address{Laboratoire de math\'ematiques (LAMA), Universit\'e de Savoie, 
Le Bourget-du-Lac}
\email{Patrick.Verovic@univ-savoie.fr}

\thanks{The first author was partially supported by the FNS grant no.\ 20-137696/1. The third author was partially supported by the CONICYT grant no.\ 3120071}
\begin{document}

\begin{abstract}
 We study the spectrum of the Finsler--Laplace operator for regular Hilbert geometries, defined by convex sets with $C^2$ boundaries. We show that for an $n$-dimensional geometry, the spectral gap is bounded above by $(n-1)^2/4$, which we prove to be the infimum of the essential spectrum. We also construct examples of convex sets with arbitrarily small eigenvalues. 
\end{abstract}
\maketitle

\tableofcontents

\section{Introduction}

Hilbert geometries, introduced by David Hilbert to illustrate the fourth of his twenty-three problems, are among the most simple and studied examples of Finsler geometries. They can be considered as a generalization of hyperbolic geometry in the context of metric geometry, and a general and now well studied question is to understand if they inherit the same geometric or analytic properties as the hyperbolic space; see for instance \cite{HandbookHilbert} for a good overview.\\
In \cite{moi:natural_finsler_laplace}, the first author introduced and began to study a new generalization of the Laplace operator to Finsler geometry. It thus gives another analytical tool to understand the differences between Hilbert geometries and the hyperbolic space. For the $n$-dimensional hyperbolic space, the spectrum of the Laplace operator is known to be the interval $[(n-1)^2/4,\infty)$. In particular, it consists only of its essential part, and there is no eigenvalue below $(n-1)^2/4$ (see for example \cite{Donnelly_essential_spectrum}). In this article, we will see that the bottom of the essential spectrum of a regular $n$-dimensional Hilbert geometry is also $(n-1)^2/4$, but that, in contrast with hyperbolic geometry, a lot of arbitrarily small eigenvalues could appear under the essential spectrum.

\subsection{Finsler and Hilbert metrics}

\begin{defin} \label{def:finsler_metric}
Let $M$ be a manifold. A \emph{Finsler metric} on $M$\ is a continuous function ${F \colon TM \rightarrow \R^+}$ that is:
\begin{enumerate}
  \item $C^{2}$, except on the zero section;
  \item positively homogeneous, that is, $F(x,\lambda v)=\lambda F(x,v)$\ for any $\lambda\geqslant 0$;
  \item positive-definite, that is, $F(x,v)\geq0$\ with equality iff $v=0$;
  \item strongly convex, that is, $ \left(\dfrac{\partial^2 F^2}{\partial v_i \partial v_j}\right)_{i,j}$ is positive-definite.
 \end{enumerate}
\end{defin}

A \emph{Hilbert geometry} is a metric space $(\C,\d)$ where
\begin{itemize}
 \item $\C$ is a properly convex open subset of the projective space $\RP^n$; \emph{properly convex} means that $\C$ contains no affine line; in other words, it appears as a relatively compact open set in some affine chart.
 \item $\d$ is a metric on $\C$ is defined in the following way (see Figure \ref{fig_hilbert_distance}): for $x,y \in \C$, let $a$ and $b$ be the intersection points of the line $(xy)$ with $\dC$; then
\begin{equation*}
 d_{\C}(x,y) = \frac{1}{2} |\ln [a:b:x:y]|,
\end{equation*}
where $[a:b:x:y]$ is the cross-ratio of the four points; if we identify the line $(xy)$ with $\R\cup\{\infty\}$, it is defined by $[a:b:x:y]=\frac{|ax|/|bx|}{|ay|/|by|}$ . 
\end{itemize}
When $\C$ is an ellipsoid, the Hilbert geometry of $\C$ gives the Klein--Beltrami model of hyperbolic space.\\

The Hilbert metric $\d$ is generated by a field of norms $\F$ on $\C$, i.e., $\d(x,y)=\inf \int_{0}^{1} F(c(t), c'(t))\ dt$, where the infimum is taken over all $C^1$ curves $c\colon [0,1]\longrightarrow \C$ from $x$ to $y$. In an affine chart containing $\C$ as a relatively compact subset, the norm $F(x,u)$ of a tangent vector $u\in T_x\C$ is given by the formula
\begin{equation*}
 F_{\C}(x,u) = \frac{|u|}{2} \left(\frac{1}{|u^-x|} +\frac{1}{|xu^+|}\right),
\end{equation*}
where $|\ \cdot\ |$ is an arbitrary Euclidean metric on the affine chart, and $u^+$ and $u^-$ are the intersection points of the line $x+\R.u$ with the boundary $\dC$ (see Figure \ref{fig_finsler_metric}).

\begin{figure}[h]
\begin{subfigure}[b]{0.48\textwidth}
\centering
 \includegraphics[width=0.8\textwidth]{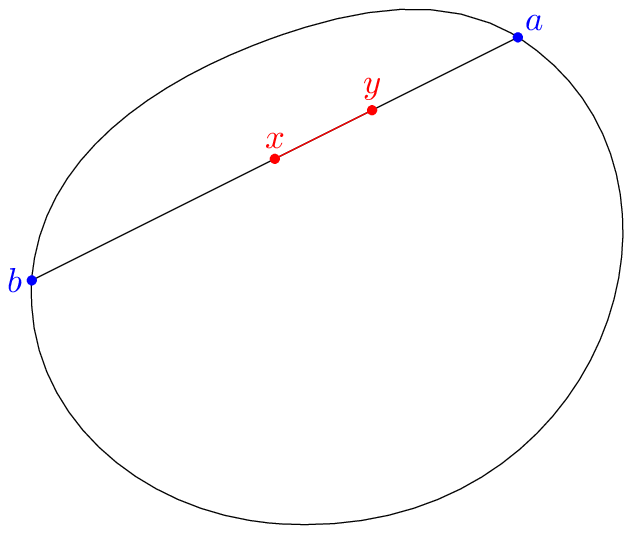}
\caption{${d_{\C} (x,y) = \left|\ln [a:x:y:b]\right|/2 }$}
\label{fig_hilbert_distance}
\end{subfigure}
\begin{subfigure}[b]{0.48\textwidth}
\centering
\includegraphics[width=0.8\textwidth]{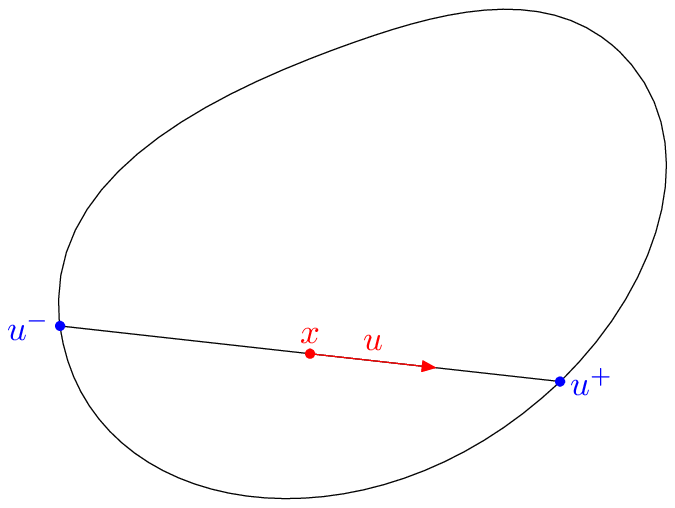}
\caption{${F_{\C} (x,u) =\left(1/|u^-x| + 1/|xu^+|\right)|u|/2}$}
\label{fig_finsler_metric} 
\end{subfigure}
\caption{}
\end{figure}

In general, a Hilbert geometry fails to be a Finsler space due to regularity issues: the regularity of $\F$ depends on the boundary of $\C$, so $\F$ does not necessarily satisfy the first and fourth points of Definition \ref{def:finsler_metric}. However, when $\C$ has a $C^2$ boundary with positive definite Hessian (see section \ref{sec_hessian}), $F_{\C}$ is a Finsler metric. In this case, the Hilbert geometry is called \emph{regular} and we can prove that its flag-curvature is constant equal to $-1$ (\cite{Fou:EquaDiff}).

\subsection{Main results}

The definition of the Finsler--Laplace operator is recalled in section \ref{sec:operator}. As for the Riemannian one, it is an unbounded elliptic operator on a Sobolev space contained in the $L^2$ functions. As such, the Finsler Laplacian admits a spectrum which splits into a discrete part, which, if non-empty, consists only of eigenvalues of finite multiplicity, and the essential spectrum. In the case at hand, there will always be an essential spectrum as we are considering non-compact manifolds.

In hyperbolic space, the spectrum of the Laplace--Beltrami operator is the interval $\left[(n-1)^2/4, +\infty \right)$ and, therefore, has no discrete part. In the case of regular Hilbert geometries, we prove the following:
\begin{thmintro} \label{thmintro_bottom}
  Let $\lambda_1(\C)$ be the bottom of the spectrum of the Finsler Laplacian of a regular Hilbert geometry $(\C,\d)$. Then
\begin{equation*}
 0 < \lambda_1(\C) \leqslant  \frac{(n-1)^2}{4}.
\end{equation*}
\end{thmintro}

Let us make some remarks about this theorem.
\begin{itemize}
 \item A study of spectral gaps in (regular and non-regular) Hilbert geometries was already launched by the second author and C.\ Vernicos \cite{ColboisVernicos_bas_du_spectre,Vernicos:spectral_radius_hilbert}. The spectral gap they were considering turns out to be associated, in the regular case, to the \emph{non-linear} Laplacian introduced by Z.\ Shen \cite{Shen:non-linear_Laplacian}, and their techniques and difficulties differ from ours. In particular, in \cite{Vernicos:spectral_radius_hilbert}, Vernicos proves that the spectral gap he considers is also less than $(n-1)^2/4$, but the difficulties for his proof appear only when considering non-regular Hilbert metrics, contrarily to us.
 \item For regular Hilbert metrics the volume entropy is always equal to $n-1$ \cite{ColboisVerovic:hilbert_geometry}. So, Theorem \ref{thmintro_bottom} in particular tells us that the inequality $4\lambda_1 \leqslant h^2$, which is true for all simply connected non-positively curved Riemannian manifolds, still holds for regular Hilbert geometries.
 \item In \cite{moi:these} the first author proved that, for negatively curved Finsler manifolds, the inequality $4\lambda_1 \leqslant n h^2$ holds, where $n$ is the dimension of the manifold. For general non-compact negatively curved Finsler manifolds, it is far from clear that the factor $n$ can be removed. In this article, we prove it for what we call \emph{asymptotically Riemannian Finsler metrics} of which Hilbert metrics are a nice example. This means that the Finsler metric gets infinitely close to Riemannian outside sufficiently big compact sets (see Section \ref{sec_bottom_of_spectrum}).
\end{itemize}

Our second result shows that the difference between regular Hilbert geometry and hyperbolic geometry does not appear in the essential spectrum (or, at least, not in its infimum):

\begin{thmintro} \label{thmintro_essential_bottom}
The bottom of the essential spectrum $\inf \sigma_{\textrm{ess}}(\C)$ of the Finsler Laplacian of a regular Hilbert geometry $(\C,\d)$ satisfies
\begin{equation*}
\inf \sigma_{\textrm{ess}}(\C) =  \frac{(n-1)^2}{4}.
\end{equation*}
\end{thmintro}

Below $\frac{(n-1)^2}{4}$, the spectrum of the Laplace operator is thus entirely discrete. It is then natural to ask if there is always an eigenvalue below $\frac{(n-1)^2}{4}$. We know that this does not happen in the hyperbolic space and we make the following
\begin{conjecture*}
Let $(\C,\d)$ be a regular Hilbert geometry. The equality $\lambda_1(\C) =  \frac{(n-1)^2}{4}$ holds if and only if $\C$ is an ellipsoid.
\end{conjecture*}

We are not yet able to prove this conjecture, but we show the following:

\begin{thmintro} \label{thmintro_small_eigenvalues}
Let $\eps >0$ and $N \in \N$. There exists a regular Hilbert geometry whose first $N$ eigenvalues are below $\eps$.
\end{thmintro}

In particular, we can find a regular Hilbert geometry with as many eigenvalues below the essential spectrum as we want. As the flag curvature of regular Hilbert metrics is always equal to $-1$, this gives examples of Finsler metrics of constant negative curvature with eigenvalues as small as we want.

\subsection*{Structure of this paper}
In the preliminaries, we recall the construction of the Finsler Laplacian and its basic properties. We also introduce the Legendre transform that will be an important tool all along the article.\\
In Section \ref{sec_behavior_at_infinity_regular_hilbert}, we prove that regular Hilbert geometries are asymptotically Riemannian, which is an interesting result in itself. \\
In Section \ref{sec_bottom_of_spectrum}, we prove Theorem \ref{thmintro_bottom} by showing that the inequality $\lambda_1 \leqslant h^2/4$ holds for asymptotically Riemannian metrics.\\
After recalling a few results about the essential spectrum of weighted Laplacians, we prove Theorem \ref{thmintro_essential_bottom} in Section \ref{sec_bottom_of_essential_spectrum}.\\
We finally construct Hilbert metrics with arbitrarily many small eigenvalues in Section \ref{sec_small_eigenvalues}.

\section{Preliminaries}

\subsection{Topology on the set of Finsler metrics on a manifold}\label{sec:topology}

In all the text, we will use the topology of uniform convergence on compact sets for Finsler metrics. Let $M$ be a smooth manifold. We say that a sequence of Finsler metrics $(F_n)$ on $M$ converges to the Finsler metric $F$ if, for any compact subset $K$ of $M$, 
$$
\lim_{n\to +\infty} \sup_{(x,u)\in TM|_K} \left|\ln \frac{F(x,u)}{F_n(x,u)}\right| = 0,
$$
where $TM|_K$ is the restriction of the tangent bundle to $K$.
This induces a topology on the set of Finsler metrics on $M$, which is metrizable: a distance between $F$ and $F'$ can be defined as
$$
d(F,F') = \sum_n \frac{1}{2^n} \min\left\{1, \sup_{(x,u)\in TM|_{K_n}} \left|\ln \frac{F(x,u)}{F'(x,u)}\right|\right\},
$$
where $(K_n)$ is an exhausting family of compact subsets of $M$.

\subsection{Finsler Laplacian}\label{sec:operator}
In this section, we quickly recall the definition of the Finsler Laplacian we consider, which uses the formalism introduced by Foulon \cite{Fou:EquaDiff}. All the proofs and details can be found in \cite{moi:these,moi:natural_finsler_laplace}.

Let $M$ be an $n$-dimensional smooth manifold. Let $HM$\ be the \emph{homogeneous bundle} or direction bundle, that is, 
$$HM := \left(TM \smallsetminus \{0\} \right) / \R^+.$$
A point of $HM$ is a pair consisting in a point $x\in M$ and a tangent direction $\xi$ at $x$. We denote by $r \colon TM\smallsetminus \{0\}  \rightarrow HM$ and $\pi \colon HM \rightarrow M$\ the canonical projections. The bundle $VHM = \Ker d\pi \subset THM$ is called the \emph{vertical bundle}.

Let $F$ be a Finsler metric on $M$. As for a Riemannian space, the metric space $(M,F)$ is locally uniquely geodesic, the geodesics being defined through a second-order differential equation. We assume in the sequel that the Finsler metric is complete. In this case, its geodesic flow is well defined on the homogeneous bundle: given a point $(x,\xi)$ in $HM$, there is a unique unit-speed geodesic $c:\R\longrightarrow M$ passing at $x$ with tangent direction $\xi$ at time $0$.

The \emph{Hilbert form} of $F$ is the $1$-form $A$ on $HM$ defined, for $(x,\xi) \in HM$, $Z \in T_{(x,\xi)}HM$, by
\begin{equation}
 A_{(x,\xi)}(Z) := \lim_{\eps \rightarrow 0} \frac{F\left(x, v + \eps d\pi(Z) \right) - F\left( x,v \right)}{\eps},
\end{equation}
where $v \in T_xM$ is any vector such that $r(x,v) = (x,\xi)$. The Hilbert form contains all the necessary information about the dynamics of the Finsler metric:

\begin{thm} 
 The form $A$ is a contact form: $\ada$ is a volume form on $HM$.
 Let $X \colon HM \rightarrow THM$ be the Reeb field of $A$, that is, the only solution of 
\begin{equation}
\label{eq:Reeb_field}
\left\{ 
\begin{aligned}
  A(X) &= 1 \\
 i_X dA &= 0  \, .
 \end{aligned}
\right.
\end{equation}
The vector field $X$ generates the geodesic flow of $F$.
\end{thm}

We can now define the Finsler Laplacian. First we split the canonical volume $\ada$ into a volume form on the manifold $M$ and an angle form:
\begin{prop}
\label{prop:construction}
 There exists a unique volume form $\Omega^F$\ on $M$\ and an $(n-1)$-form $\alpha^F$\ on $HM$, never zero on $VHM$, such that
\begin{equation}
\label{eq:alpha_wedge_omega}
  \alpha^{F} \wedge \pi^{\ast}\Omega^F =  A\wedge dA^{n-1}, 
\end{equation}
and, for all $x\in M$, 
\begin{equation}
\label{eq:longueur_fibre}
 \int_{H_xM} \alpha^F =  \voleucl(\S^{n-1})\, .
\end{equation}
\end{prop}

\begin{rem}
 The volume form $\frac{1}{(n-1)!} \Omega^F$ is the Holmes--Thompson volume form (see for instance \cite{BuragoBuragoIvanov} or \cite{Alvarez:Survey} for the definition). However, we will not need in this article any specific knowledge about the Holmes--Thompson volume.
\end{rem}

The Finsler Laplacian of a function is then obtained as an average with respect to $\alpha^F$ of the second derivative in every direction:
\begin{defin}
\label{def:delta}
 For $f \in C^2(M)$, the \emph{Finsler--Laplace operator} $\Delta^F$\ is defined by
 $$
 \Delta^F f (x) = \frac{n}{\voleucl \left(\mathbb{S}^{n-1}\right) }\int_{H_xM} L_X ^2 (\pi^{\ast} f ) \alpha^F,\ x \in M,
 $$
where $L_X$ denotes the Lie derivative in the direction $X$.
\end{defin}
This definition gives a second order elliptic differential operator, which is symmetric with respect to the Holmes--Thompson volume $\Omega^F$.
The constant in front of the operator is there in order to get back the usual Laplace--Beltrami operator when $F$ is Riemannian.

The \emph{symbol} of a second-order differential operator $\Delta$ is a symmetric bilinear form on the co-tangent bundle that can be defined in the following way: Let $\xi \in T^{\ast}_xM$, then the symbol of the operator $\Delta$ at $(x,\xi)$ is 
$$
\sigma_x(\xi,\xi) = \Delta(\varphi^2)(x),
$$
where $\varphi \colon M \rightarrow \R$ is a $C^2$ function such that $\varphi(x) =0$ and $d\varphi_x= \xi$.

When the operator is elliptic, that is, when $\sigma_x(\xi,\xi) >0$ for all non-zero $\xi$, the symbol defines a dual Riemannian metric. Note that in local coordinates, the symbol is given by the matrix of the coefficients in front of the second order derivatives.

 We denote by $\sigma^F$ the symbol of $\Delta^F$, as $\Delta^F$ is elliptic, $\sigma^F$ is a dual Riemannian metric. In our case, we can express $\sigma^F$ using the form $\alpha^F$:
 For $\xi_1,\xi_2 \in T^{\ast}_x M$, we have
\begin{equation*}
 \sigma^F_x(\xi_1,\xi_2) = \frac{n}{\voleucl \left(\mathbb{S}^{n-1}\right) } \int_{H_xM} L_X(\pi^{\ast} \varphi_1) L_X(\pi^{\ast}\varphi_2)\, \alpha^F,
\end{equation*}
where $\varphi_i \in C^{\infty}(M)$ such that $\varphi_i(x)=0$ and $\left.d\varphi_i\right._x = \xi_i$. Note that, if we identify $HM$ with $S^FM$, the unitary tangent bundle for $F$, and that we consider $\alpha^F$ as a volume form on $S^FM$ (instead of $HM$), we have this visually more agreeable formula:
\begin{equation*}
 \sigma^F_x(\xi_1,\xi_2) = \frac{n}{\voleucl \left(\mathbb{S}^{n-1}\right) } \int_{v\in S^F_xM} \xi_1(v) \xi_2(v) \, \alpha^F(v).
\end{equation*}

Note that we can also see $\Delta^F$ as a weighted Laplacian (introduced in \cite{ChavelFeldman:Isoperimetric_constants,Davies:Heat_kernel_bounds}), with symbol $\sigma^F$ and weight given by the ratio between $\Omega^F$ and the Riemannian volume associated with $\sigma^F$. Indeed, we have that, if $a\in C^{\infty}(M)$ is such that $\Omega^F = a^2 \Omega^{\sigma^F}$, where $\Omega^{\sigma^F}$ is the Riemannian volume associated with $\sigma^F$, then for $\varphi \in C^{\infty}(M)$:
\begin{equation*}
 \Delta^F \varphi = \Delta^{\sigma^F} \varphi - \frac{1}{a^2} \langle \nabla \varphi , \nabla a^2 \rangle.
\end{equation*}
The description of $\Delta^F$ in terms of a weighted Laplacian will come very handy for the study of the essential spectrum in Section \ref{sec_bottom_of_essential_spectrum}.

\subsection{Energy and bottom of the spectrum} \label{subsec_energy_and_spectrum}

The Finsler Laplacian has a naturally associated \emph{energy functional} defined by
\begin{equation}
 E^F(f) := \frac{n}{\voleucl \left(\S^{n-1}\right) } \int_{HM} \left|L_X\left(\pi^{\ast}f \right)\right|^2 \ada. \label{eq_energy_finsler}
\end{equation}
The \emph{Rayleigh quotient} for $F$ is then defined by
\begin{equation}
 R^F(f) := \frac{E^F(f)}{\int_M f^2\, \Omega^F}.
\end{equation}

Let $H^1(M)$ be the Sobolev space defined as the completion of $C^{\infty}_0(M)$, the space of smooth functions with compact support, under the norm $\lVert f \rVert^2_1 = \int_M f^2\, \Omega^F + E^F(f)$.

The bottom of the spectrum of $-\Delta^F$, considered as a symmetric unbounded operator on $H^1(M)$, is given by:
\begin{equation*}
 \lambda_1 = \inf_{f \in H^1(M)} R^F(f),
\end{equation*}
 Note that, as the manifolds we are interested in in this article are not compact, the spectrum has no reason to be discrete. However, if there is a discrete spectrum below the essential one, then the eigenvalues can be obtained from the Rayleigh quotient via the Min-Max principle:
\begin{thm}[Min-Max principle] \label{thm_min_max}
Suppose that $\lambda_1$, \dots, $\lambda_k$ are the first $k$ eigenvalues (counted with multiplicity) of $-\Delta^F$ and are all below the essential spectrum, then
\begin{equation*}
 \lambda_i = \inf_{V_i} \sup \left\{ R^F(u) \mid u \in V_i \right\}
\end{equation*}
where $V_i$ runs over all the $i$-dimensional subspaces of $H^1(M)$.
\end{thm}

\subsection{Cotangent point of view}

We finish this preliminaries with the cotangent point of view for Finsler metrics. This is fairly well-known and we refer to \cite{moi:these} for a more detailed presentation.

\subsubsection{Dual metric}

\begin{defin}
  Let $F$ be a Finsler metric on a manifold $M$. The \emph{dual Finsler (co)metric} ${F^{\ast} \colon T^{\ast}M \rightarrow \R}$ is defined, for $(x,p) \in T^{\ast}M$, by
\begin{equation*} 
F^{\ast}(x,p) = \sup \lbrace p(v) \mid v\in T_xM \; \text{\rm such that } F(x,v)=1 \rbrace.
\end{equation*}
\end{defin}

\subsubsection{Legendre transform}

The tool that allows us to switch from the tangent bundle to the cotangent bundle is the \emph{Legendre transform} associated with $F$.
\begin{defin}
 The \emph{Legendre transform} $\L_F : TM \rightarrow T^{\ast}M$ associated with $F$ is defined by $\L_F(x,0) = (x,0)$ and, for $(x,v) \in TM \smallsetminus \{0\}$ and $u \in T_xM$, 
\begin{equation*}
 \L_F ( x,v) (u) := \frac{1}{2} \left. \frac{d}{dt} F^2(x,v + tu)\right|_{t=0}.
\end{equation*}
\end{defin}
As $F^2$ is positively $2$-homogeneous, we have that $\L_F$ is positively $1$-homogeneous, that is,  $\L_F (x,\l v) (u) = \l  \L_F (x,v) (u) $ for $\l\geqslant 0$. So we can project $\L_F$ to the homogeneous bundle. Set $H^{\ast}M := T^{\ast}M \smallsetminus \{0\} / \R^+_{\ast}$ and write $\ell_F\colon HM \rightarrow H^{\ast}M$ for the projection. Considering directly $\ell_F$, instead of $\L_F$, can sometimes be quite helpful.

The Legendre transform $\L_F$ links the Finsler metric $F$ with its dual metric $F^{\ast}$
\[
 F = F^{\ast} \circ \L_F.
\]
So, in particular $\L_F$ maps the unit tangent bundle of $F$ to the unit cotangent bundle of $F^{\ast}$.

Moreover, the Legendre transform $\L_F$ is a diffeomorphism and the following diagram commutes (see for instance \cite{moi:these}):
\[
\xymatrix{
    & T^{\ast}M \smallsetminus \{0\} \ar[r]^{\hat{r}} \ar[ld]_{\hat{p}}  &   H^{\ast}M \ar[rd]^{\hat{\pi}} & \\
M   &       &       &   M  \\
    & TM \smallsetminus \{0\} \ar[uu]^{\L_F} \ar[r]_r  \ar[ul]^{p} & HM \ar[uu]_{\ell_F} \ar[ur]_{\pi} &  }
\]

For strongly convex smooth Finsler metrics, the Legendre transform can also be described using convex geometry. The Legendre transform associated with a convex $\mathcal{C} \subset \R^n$ sends a point $x$ of $\mathcal{C}$ to the hyperplane supporting $\mathcal{C}$ at $x$, or equivalently, to the linear map $p \in \left(\R^n\right)^{\ast}$ such that $p(x)=1$ and $\ker p$ is parallel to the supporting hyperplane.

\subsubsection{Continuity of the Legendre transform}

Let $V$ be a $n$-dimensional real vector space\footnote{In this section, we should think of a Finsler manifold $(M,F)$ with a fixed point $x\in M$. We look at the tangent space $T_xM$ as an $n$-dimensional real vector space, provided with a non-necessarily symmetric norm $F(x,\cdot)$.}, with a fixed Euclidean structure whose norm we denote by $F_0$ and see as a translation-invariant Finsler metric on $V$.\\
Let $\mathcal{N}$ denote the set of translation-invariant Finsler metrics on $V$. This is the same as looking at the set of non-necessarily symmetric norms on $V$, whose unit sphere is $C^2$ with positive definite Hessian.\\
The topology defined in section \ref{sec:topology} induces a topology on $\n$ which can be metrized  in the following easy way: Let $HV = V\smallsetminus\{0\}/\R^+ \simeq \S^{n-1}$ be the set of rays from the origin. If $F,F'\in\n$, the ratio $\frac{F}{F'}$ is a well defined function of $HV$: if $\xi\in HV$, we have $\frac{F}{F'}(\xi)=\frac{F(u)}{F'(u)}$, where $u$ is any vector of $V$ that projects to $\xi$. Define a metric on $\mathcal{N}$ by
$$d_{\n}(F,F')= \sup_{\xi\in HV} |\ln \frac{F}{F'}(\xi)|.$$
We define a metric $D$ on the set $\Homeo^0(V)$ of positively $1$-homogeneous homeomorphisms of $V$ by:
$$D(H,H') = \sup_{u\in V,\ F_0(u)=1} F_0(H(u)-H'(u)).$$
Identifying $HV$ with the unit Euclidean sphere $\S^{n-1}$, we define a metric $d$ on the set $\Homeo(HV)$ of homeomorphisms of $HV$:
$$d(h,h') = \sup_{\xi\in HV} d_{\S^{n-1}}(h(\xi),h'(\xi)).$$
This distance is just the maximal Euclidean angle between the images.\\

For each $F\in \mathcal{N}$, the Legendre transform $\L_F$ is a positively $1$-homogeneous homeomorphism of $V$ and its ``projection'' $\ell_F$ a $C^1$-diffeomorphism of $HV$. We thus have applications $\L:F \longmapsto \L_F$ from $\n$ to $\Homeo^0(V)$ and $\ell:F \longmapsto \ell_F$ from $\n$ to $\Homeo(HV)$. The following lemma is immediate if we use the geometrical interpretation of the Legendre transform that we recalled at the end of the previous section.

\begin{lem}\label{lem_continuity_l}
The application $\L$ is a continuous bijection from $(\n,d_{\n})$ to $(\Homeo^0(V),D)$. The application $\ell$  is continuous from $(\n,d_{\n})$ to $(\Homeo(HV),d)$ but is not injective: $\ell_F=\ell_{F'}$ if and only if $F = \lambda F'$ for some $\lambda>0$.
\end{lem}

\begin{proof}
Let us explicit the continuity of $\ell$ at $F_0$ because this is all we need in this article; the continuity elsewhere follows the exact same lines.\\
Let $F\in\mathcal{N}$ such that $d_{\mathcal{N}}(F_0,F) \leqslant \ln C$ for some $C>1$. We can see that $d(\ell_{F}^{-1}\circ \ell_{F_0}, \id) \leqslant \arccos C^{-2}$.\\
Indeed, as $d_{\mathcal{N}}(F_0,F) \leqslant \ln C$, the unit sphere $S_F(1)$ for $F$ in $V$ is in between the spheres of radius $C^{-1}$ and $C$ for $F_0$, that we denote by $S_0(C^{-1})$ and $S_0(C)$. Let $\xi \in HV$. The map $\ell_{F}^{-1}\circ \ell_{F_0}$ sends $\xi$ to a point $\xi'$, such that the tangent space of $S_F(1)$ at $\xi'$ is parallel to the tangent space of $S_0(C)$ at $\xi$. Figure \ref{fig_the_control_legendre_transform} and simple trigonometry then yield the result.
\begin{figure}[h]
\begin{center}
\begin{pspicture}(0,-2.7)(6.62,3)
\pscircle[linewidth=0.04,dimen=outer](2.3,-0.4){2.3}
\pscircle[linewidth=0.04,dimen=outer](2.3,-0.4){1.7}
\psline[linewidth=0.04cm](1.2,2.3)(5.1,-0.2)
\psline[linewidth=0.04cm](2.4,-0.5)(3.8,2.0)
\psline[linewidth=0.04cm,linestyle=dashed,dash=0.16cm 0.16cm](2.4,-0.5)(5.2,0.395)
\psline[linewidth=0.04cm,linestyle=dashed,dash=0.16cm 0.16cm](2.4,-0.5)(1.8,2.5)
\psbezier[linewidth=0.04,linecolor=red](3.6559284,-2.02)(3.3135412,-2.22)(2.548205,-2.48)(1.8030093,-2.38)(1.0578135,-2.28)(0.68,-1.82)(0.49388167,-1.02)(0.30776334,-0.22)(0.789305,1.2040415)(1.7224476,1.58)(2.6555903,1.9559585)(3.2101226,1.3714887)(3.6156476,0.78)(4.0211725,0.1885113)(4.239899,-0.4818077)(4.21986,-0.92)(4.1998215,-1.3581923)(3.9983156,-1.82)(3.6559284,-2.02)
\psline[linewidth=0.04cm,linecolor=red](4.44,0.56)(1.46,2.46)
\psline[linewidth=0.04cm,linecolor=red](4.44,0.56)(1.46,2.46)
\psline[linewidth=0.04cm,linecolor=red,linestyle=dashed,dash=0.16cm 0.16cm](2.4,-0.5)(3.22,2.7)
\uput{3pt}[45](3.8,2.0){$\xi$}
\rput(2.6,-1.6){$S_0( C^{-1})$}
\put(4,-2.3){$S_0( C)$}
\put(2.7,2.9){\color{red} $\ell_{F}^{-1} \circ \ell_{F_0} \left( \xi \right)$}
\put(1.8,-2.4){\color{red} $S_F(1)$}
\end{pspicture} 
\quad
\begin{pspicture}(1,-1.6)(7,2.4)
\psarc[linewidth=0.04,linestyle=dashed,dimen=outer](2.4,-0.5){1.7}{-15}{100}
\psarc[linewidth=0.04,linestyle=dashed,dimen=outer](2.4,-0.5){2.3}{-15}{100}
\psline[linewidth=0.04cm](1.2,2.3)(5.1,-0.2)
\psline[linewidth=0.04cm](2.4,-0.5)(3.8,2.0)
\psline[linewidth=0.04cm](2.4,-0.5)(5.2,0.25)
\psarc[linewidth=0.03cm]{->}(2.4,-0.5){0.6}{17}{59}
\rput(2,0.9){$S_0(C^{-1})$}
\rput(4.7,-1.3){$S_0(C)$}
\end{pspicture} 
\end{center}
\caption{Maximum angle between $\ell_{F}^{-1} \circ \ell_{F_0} \left( \xi \right)$ and $\xi$}
 \label{fig_the_control_legendre_transform}
\end{figure}
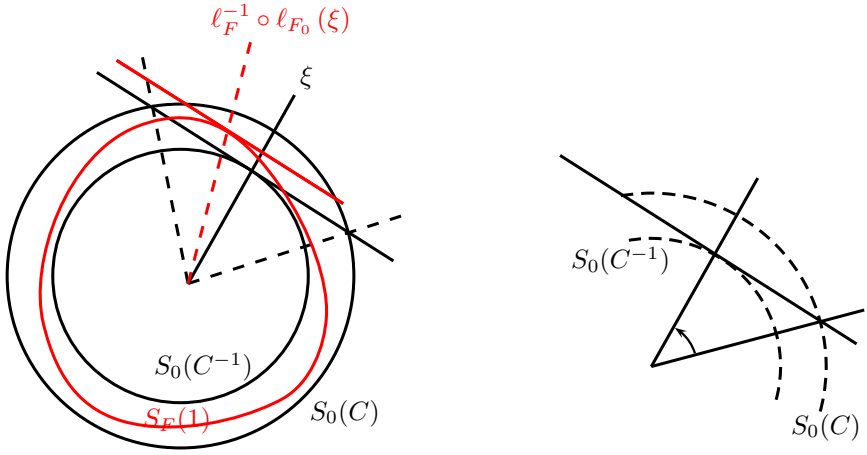
\end{proof}

\section{Behavior at infinity of Regular Hilbert geometries}\label{sec_behavior_at_infinity_regular_hilbert}

In all the following, $(\C, d_{\C})$ will be a regular Hilbert geometry. We will see here that $(\C, d_{\C})$ is asymptotically Riemannian, that is, the space looks more and more like a Riemannian space outside big compact sets:

\begin{defin} \label{def_asympt_riemannian}
 A Finsler metric $F$ on a manifold $M$ is called \emph{asymptotically Riemannian} if, for any $C>1$, there exists a compact set $K$ such that, for all $x \in M \smallsetminus K$, there exists a scalar product $g_x$ on $T_xM$ satisfying, for every non-zero vector $v\in T_xM$,
\begin{equation*}
 C^{-1} \leqslant \frac{F(x,v)}{\sqrt{g_x(v,v)}} \leqslant C
\end{equation*}
\end{defin}

\begin{rem}
Note that, for this definition to be of any interest, $M$ should be non-compact.
\end{rem}

\subsection{Hessian of a codimension-$1$ submanifold of the projective space}\label{sec_hessian}

Consider a codimension-$1$ $C^2$ submanifold $N$ of the projective space $\RP^n$ (for instance the boundary $\dC$ of a convex set $\C$), and pick a point $x\in N$. Choose an affine chart containing $x$ and a Euclidean metric on it. Let $n$ be a unit normal vector to $N$ at $x$ for this metric, that is a unit vector orthogonal to $T_xN$. Now, around $x$, we consider $N$ as the graph of the function, defined on some neighborhood $U$ of $x$ in $T_xN$:
$$G_x \colon u\in U \longmapsto G_x(u)\in\R,$$
such that a neighborhood of $x$ in $N$ is the submanifold $\{u+G(u).n,\ u\in U\}$. The Hessian of $G_x$ at $x$ is a bilinear form on the tangent space $T_xN$. If one chooses an orthonormal basis $(u_1,\cdots,u_{n-1})$ of $T_xN$, then the matrix of this bilinear form is the $(n-1)\times(n-1)$ matrix $\left(\frac{\partial^2 G}{\partial u_i\partial u_j}\right)$ of the second-derivatives of $G$.\\

The definition of the Hessian obviously depends on the choice of the affine chart and of the Euclidean metric. Nevertheless, there are two basic observations which we will use all along this section.

\begin{itemize}
 \item The property of the Hessian of $N$ at $x$ to be positive, negative or definite, is independent of the choice of the affine chart and the Euclidean metric. Hence, for example, it is possible to talk about a convex subset of $\RP^n$ whose boundary is $C^2$ with positive definite Hessian.
 \item Let $N'$ be another codimension-$1$ $C^2$ submanifold of $\RP^n$, which is tangent to $N$ at $x$. It makes sense to say that $N$ and $N'$ have the same Hessian at $x$. Indeed, choose an affine chart containing $x$, a Euclidean metric on it, a unit vector $n$ normal to $N$ at $x$ and an orthonormal basis $(u_1,\cdots,u_{n-1})$ of $T_xN=T_xN'$. Call $H_x$ and $H_x'$ the Hessians of $N$ and $N'$ at $x$. The fact that they are the same bilinear form on $T_xN$ does not depend on any of the previous choices. 
\end{itemize}

\subsection{Busemann functions, horospheres and horoballs}

The \emph{Busemann function} based at $x\in\dC$ is defined by
$$b_{x}(z,y) = \lim_{p\to x} d_{\C}(z,p) - d_{\C}(y,p),$$
which, in some sense, measures the (signed) distance from $z$ to $y$ in $\C$ as seen from the point $x\in\dC$. A particular expression for $b$ is given by
$$b_{x}(z,y)=\lim_{t\to +\infty} d_{\C}(z,\gamma(t))-t,$$
where $\gamma$ is the geodesic leaving $y$ at $t=0$ to $x$. When $x$ is fixed, then $b_{x}$ is a surjective map from $\C\times\C$ onto $\R$. When $z$ and $y$ are fixed, then $b_.(z,y): \dC \rightarrow \R$ is bounded by a constant $C=C(z,y)$.\\

The \emph{horosphere} passing through $z\in \C$ and based at $x\in\dC$ is the set
$$\mathcal{H}_{x}(z)=\{y\in\C,\ b_{x}(z,y)=0\}.$$
$\mathcal{H}_{x}(z)$ is also the limit when $p$ tends to $x$ of the metric spheres $S(p,d_{\C}(p,z))$ about $p$ passing through $z$. In some sense, the points on $\mathcal{H}_{x}(z)$ are those which are as far from $x$ as $z$ is.\\
The (open) horoball $H_{x}(z)$ defined by $z\in \C$ and based at $x\in\dC$ is the ``interior'' of the horosphere $\mathcal{H}_{x}(z)$, that is, the set $$H_{x}(z)=\{y\in\C,\ b_{x}(z,y)> 0\}.$$
For example, if $\E$ is an ellipsoid, then the horoballs of $(\E, d_{\E})$ are also ellipsoids. We explain this fact in the proof of the following lemma. This proof will introduce the main construction which helps us in understanding the asymptotic behavior of Hilbert geometries.

\begin{lem}\label{busemann}
Let $(\C,\d)$ be a regular Hilbert geometry. 
\begin{itemize}
 \item For any $x \in \dC$, the Busemann function $b_{x} \colon \C\times\C \rightarrow \R$ is a $C^2$ function. 
\item Let $x\in\dC$, $z\in \C$. The set $H_x(z)\cup \{x\}$ is a $C^2$ submanifold of $\RP^n$, whose Hessian at $x$ is the same as the Hessian of $\dC$.
\end{itemize}
\end{lem}
\begin{proof}
The first point follows from the following description of the Busemann function $b_{x}(z,y)$, given by Benoist in \cite{Benoist:convexes_div_I}. Let $z'$ and $y'$ be the intersection points of the lines $(xz)$ and $(xy)$ with $\dC$, which are distinct from $x$. Let $p$ be the intersection point of $(y'z')$ with $T_x\dC$. Then 
$$b_{x}(z,y) = \frac{1}{2} \ln [(px):(pz'):(pz):(py)],$$
where $[(px):(pz'):(pz):(py)]$ denotes the cross-ratio of the four lines $(px), (pz'), (pz), (py)$. All these constructions involve only the boundary of $\C$, so the Busemann function has the same regularity as $\dC$. This first point implies that horospheres are $C^2$ submanifolds of $\C$.\\

To prove the second point, we first consider the case of an ellipsoid $\E$. The Hilbert geometry $(\E,d_{\E})$ is a model of the Riemannian hyperbolic space. We will exploit the fact that, for any $x\in\partial\E$, $\partial\E$ or any horosphere based at $x$ is an orbit of a maximal parabolic group of isometries fixing $x$. We have to prove that the Hessians are the same in all the directions, so we can assume the dimension is $2$.\\
Let then $\E$ be an ellipsoid in $\RP^2$ and pick $x\in\partial\E$. We can choose a projective basis $(e_1,e_2,e_3)$ such that $e_1=x$, $e_2\in T_x\partial\E$ and the maximal parabolic group of isometries fixing $x$ is given by $\mathcal{P} = \{g_t\in \mathrm{SL}(3,\R),\ t\in\R\}$ with
$$\left(\begin{array}{ccc} 
1 & t & \frac{t(t-1)}{2}\\
0 & 1 & t \\
0 & 0 & 1 
\end{array}\right).$$
The boundary $\partial\E$, as well as any horosphere $\mathcal{H}$ based at $x$ is the $\mathcal{P}$-orbit of the point $z=e_1+se_3$ for some $s\in \R\smallsetminus\{0\}$, that is, an ellipse parametrized by
$$t \longmapsto [1+s\frac{t(t-1)}{2}:st:s].$$
In the affine chart given by the intersection with the plane $\{(x_1,x_2,x_3)\in\R^3,\ x_1=1\}$, with origin $x$ and the induced Euclidean metric of $\R^3$, this is the curve
$$t \longmapsto \left(\frac{t}{\frac{1}{s}+\frac{t(t-1)}{2}},\frac{1}{\frac{1}{s}+\frac{t(t-1)}{2}} \right).$$
By making the transformation $t \longmapsto 1/t$, this becomes the curve
$$c \colon t \longmapsto \left(\frac{t}{\frac{t^2}{s}+\frac{1-t}{2}},\frac{t^2}{\frac{t^2}{s}+\frac{1-t}{2}} \right),$$
such that $c(0)=x$. But for $t$ around $0$, we have up to order $2$:
$$c(t) \sim (2t(1+t),2t^2).$$
This implies that the curvature of the curve $c$ at $0$ is independent of $s$, and hence, that, for ellipsoids, the Hessian of the horospheres are all the same at $x$.\\

Now, let $(\C,\d)$ be a regular Hilbert geometry. Pick $x\in\dC$, and $z\in\C$. Fix an affine chart centered at $x$, containing $\overline\C$, and fix a Euclidean metric $|\cdot|$ on it such that $|zx|=1$ and the Hessian $B_x$ of $\dC$ at $x$ is the restriction of the Euclidean scalar product to $T_x\dC$.\\
Fix $C>1$. Consider the Euclidean spheres $S_x^+$ and $S_x^-$, whose boundaries are tangent to $\dC$ at $x$ and Hessians $B_x^+$ and $B^-_x$ at the point $x$, seen as elements of $\mathrm{GL}(n-1,\R)$, satisfy $B_x^- = C B_x$ and $B_x^+ = C^{-1} B_x$ (this does not depend on the Euclidean metric we use to compute them). For $C$ close enough to $1$, the balls $\E_x^-$ and $\E_x^+$ they define contain the point $z$.\\
Let $\mathtt{h}_x^-$ and $ \mathtt{h}^+_x$ be the hyperbolic metrics defined by the balls $\E_x^-$ and $\E_x^+$. There is some neighborhood $U$ of $x$ in $\R^n$, depending on $C$, such that, on $ U\cap \E^-_x$, we have
$$\mathtt{h}^+_x\leqslant F_{\C} \leqslant \mathtt{h}^-_x.$$
Denote by $H^-_x(z)$, $H^+_x(z)$ and $H_x(z)$ the horoballs based at $x$ passing through $z$ for the Hilbert geometries defined by $\E_x^-$, $\E_x^+$ and $\C$ respectively. The previous inequality implies that
$$H^-_x(z)\cap U\subset H_x(z)\cap U\subset H^+_x(z)\cap U.$$
Now, by the result for ellipsoids, the Hessians $B_x^+(z)$ and $B^-_x(z)$ at the point $x$ of the boundaries of $H^+_x(z)$ and $H^-_x(z)$ also satisfy $B_x^-(z) = C B_x$ and $B_x^+(z) = C^{-1} B_x$. This means the horospheres $\mathcal{H}^-_x(z)$ and $\mathcal{H}^+_x(z)$ are ``almost'' osculating for $\mathcal{H}_x(z)$. Since $C>1$ is arbitrary, we see that the horosphere $\mathcal{H}_x(z)$ and $\dC$ have the same Hessians at $x$.
\end{proof}

\subsection{Hilbert geometries are asymptotically Riemannian}

\begin{prop}\label{prop_hilbert_are_bilipschitz_at_infinity}
Let $(\C,\d)$ be a regular Hilbert geometry, fix a point $o\in\C$ and a constant $C>1$. To each $x\in\dC$, we can associate a (non-complete) Riemannian hyperbolic metric $\h_x$ on $\C$ such that
\begin{enumerate}
 \item the application $x\longmapsto \h_x$ is continuous;
 \item the metric $\mathtt{h}_x$ has the same geodesics as $F_{\C}$ on $\C$;
 \item there is $R=R(C) \geqslant 0$ such that, for any $x\in\dC$ and $z\in [ox)\smallsetminus B(o,R)$,
$$C^{-1} \leqslant  \frac{F_{\C}(z, \cdot)}{\mathtt{h}_x(z, \cdot)} \leqslant C.$$
\end{enumerate}
\end{prop}
\begin{proof}
We  more or less repeat the construction used in Lemma \ref{busemann}. By choosing an adapted affine chart, we look at $\C$ as a relatively compact subset of a Euclidean space $\R^n$, with norm $|\cdot|$.\\
Let $x\in \dC$. The Hessian of $\dC$ at $x$, computed with respect to the metric $|\cdot|$, gives a positive definite bilinear form $B_x$ on $T_x\dC$, and the map $x\longmapsto B_x$ is continuous. Define a new Euclidean norm $|\cdot|_x$ on $\R^n$ by setting: 
\begin{itemize}
 \item the vector $ox$ has norm $1$: $|ox|_x=1$;
 \item the restriction of the corresponding scalar product to $T_x\dC$ is $B_x$;
 \item $T_x\dC$ and $ox$ are orthogonal. 
\end{itemize}
The map $x\longmapsto |\cdot|_x$ is continuous. The sphere $S_x$ of radius $1$ for the norm $|\cdot|_x$, with center $o$, is tangent to $\C$ at $x$; in fact, it is an osculating sphere.\\

Let $\varepsilon>0$, and consider the spheres $S_x^+$ and $S_x^-$ of respective radius $1+\varepsilon$ and $1-\varepsilon$ for $|\cdot|_x$, whose boundaries are tangent to $\dC$ at $x$. Their centers are on the line $(ox)$. Their Hessians $B_x^+$ and $B^-_x$, seen as elements of $\mathrm{GL}(n-1,\R)$, at the point $x$ satisfy $B_x^+ = \frac{1-\varepsilon}{1+\varepsilon} B_x^-$ (and this does not depend on the Euclidean metric we use to compute them).\\
Now, let $\E_x^+$ be the smallest ellipsoid which contains $\C$, has $x$ in its boundary, and such that $S_x^+$ is a horosphere at $x$ of the hyperbolic geometry defined by $\E_x^+$. In other words, it is the smallest ellipsoid which contains $\C$, is tangent to $\dC$ at $x$, has its center on $(ox)$ and the Hessian of its boundary at $x$ is the same as the Hessian of $S_x^+$. Such an ellipsoid exists in the projective space because locally around $x$, $S_x^+$ contains $\C$. However, it might not be an ellipsoid in the affine chart, but could for instance be a paraboloid or a hyperboloid.\\
In the same way, let $\E_x^-$ be the largest horosphere at $x$ of the Hilbert geometry defined by $S_x^-$ which is contained in $\C$. We also have that the Hessian of the boundary of $\E_x^-$ at $x$ is the same as the Hessian of $S_x^-$.\\
Let $\mathtt{h}_x^-$ and $ \mathtt{h}^+_x$ be the hyperbolic metrics defined by the ellipsoids $\E_x^-$ and $\E_x^+$. By definition, we have that, on $\E^-_x$,
$$\mathtt{h}^+_x\leqslant F_{\C} \leqslant \mathtt{h}^-_x.$$

We will prove that, for $\varepsilon$ small enough, the application $x\longmapsto \h_x^+$ satisfies the desired properties. The property (2) is obvious. To prove (1), we show the following

\begin{lem}\label{lem:continuity}
The maps $x\longmapsto \E^{\pm}_x$ are continuous.
\end{lem}
\begin{proof}
We show the continuity of $x\longmapsto\E_x^-$ at a given point $x_0\in\dC$, the same works for $x\longmapsto\E_x^+$. Choose a point $p$ in $\E_{x_0}^-$ and let $r\in\R$ such that $\E_{x_0}^-$ is the horoball
$$\E_{x_0}^-=\{z\in\C,\ b_{x_0}(o,z)>r\}$$
in the hyperbolic geometry defined by $S^+_{x_0}$. For any $\delta\in\R$, let $\E_x^-(\delta)$ be the (open) horoball
$$\E_x^-(\delta)=\{z\in\C,\ b_x(o,z)>r+\delta\}$$
in the hyperbolic geometry defined by $S_x^+$. For any $\delta >0$, the maps $x\longmapsto \E_x^-(\delta)$ are continuous, because of the continuity of the Busemann functions. Fix $\delta>0$. The horoball $\E_{x_0}^-(\delta)$ is entirely contained in $\C$ while the horoball $\E_{x_0}^-(-\delta)$ has a nonempty intersection with $\RP^n\smallsetminus\C$. By continuity of $x\longmapsto \E_x^-(\delta)$, the same is true for $\E_x^-(\delta)$ and $\E_x^-(-\delta)$ for $x$ in some neighborhood of $x_0$. By definition of $\E_x^-$, this implies that $\E_x^-(\delta) \subset \E_x^- \subset \E_x^-(-\delta)$ in this neighborhood, hence the continuity of $x\longmapsto\E_x^-$ at $x_0$.
\end{proof}

To prove the third point, we consider, for $x\in\dC$ and $u\in\R^n\smallsetminus\{0\}$, the function 
$$f_{x,u}: r \longmapsto \frac{\h^-_x(o_r,u)}{\h^+_x(o_r,u)},$$
where $o_r$ is the point of $[ox)$ such that $\d(o_r,o)=r$. The function $f_{x,u}$ is defined as soon as $r$ is big enough for $o_r$ to be in $\E_x^-$. Remark that $f_{x,u}=f_{x,\l u}$ for all $u\in\R^n\smallsetminus\{0\},\ \l\not= 0$.

\begin{lem}
For $u\in \R.ox\smallsetminus\{0\}$, the function $f_{x,u}$ is decreasing and tends to $1$. For $u\in T_x\dC\smallsetminus\{0\}$, the function $f_{x,u}$ is decreasing and tends to $\sqrt{\frac{1+\varepsilon}{1-\varepsilon}}$.
\end{lem}
\begin{proof}
We can choose another affine chart, with coordinates $(t_1, \dots , t_{n-1}, s)$, so that the boundary of $\E_x^+$ is the parabola $s=|t|^2$, where $|t|^2 = t_1^2 + \dots +t_{n-1}^2$. In that chart, the boundary of $\E_x^-$ has to be an ellipsoid inside of $\E_x^+$, and the line $(ox)$, which is an axis of symmetry for both $\E_x^+$ and $\E_x^-$, is sent to the $y$-axis. The equation of $\E_x^-$ is then given by 
$$\frac{s^2}{a^2}-\frac{2s}{a}+ \frac{|t|^2}{b^2}=0,$$
for some $a,b>0$.

Let $s=s(r)=|o_rx|$. If $u\in\R.ox\smallsetminus\{0\}$, we have (see Figure \ref{fig_E+_et_E-_dans_carte})
\begin{equation*}
 \h^-_x(o_r,u)= \frac{a|u|}{s(2a-s)}, \quad \text{and} \quad \h^+_x(o_r,u) =  \frac{|u|}{2s}.
\end{equation*}
Hence
$$f_{x,u}(r) = \frac{2a}{2a-s(r)}$$
which is decreasing and tends to $1$.\\
If $u\in T_x\dC$, then (see Figure \ref{fig_E+_et_E-_dans_carte})
\begin{equation*}
 \h^-_x(o_r,u)= \frac{a|u|}{b\sqrt{s(2a-s)}}, \quad \text{and} \quad \h^+_x(o_r,u) =  \frac{|u|}{\sqrt{s}}.
\end{equation*}
Hence
$$f_{x,u}(r) = \frac{a}{b\sqrt{2a-s(r)}},$$
which is decreasing and tends to $\frac{\sqrt{a}}{b\sqrt{2}}$. Now, direct computations shows that in this chart, $B_x^- = a/b^2$ and $B_x^+ =2$, hence, $f_{x,u}(r)$ converges to $\sqrt{B_x^-(B_x^+)^{-1}} = \sqrt{\frac{1+\varepsilon}{1-\varepsilon}}>1$.
\begin{figure}[h]
 \centering
\begin{pspicture}(0,-3.42)(10.52,3.62)
\psline[linewidth=0.04cm,arrowsize=0.25]{->}(0.0,-3.0)(10.0,-3.0)
\psline[linewidth=0.04cm,arrowsize=0.25]{->}(5.0,-3.0)(5.0,3.6)
\psbezier[linewidth=0.04](0.2,2.6)(1.2,-1.6)(2.6,-3.0)(5.0,-3.0)(7.4,-3.0)(8.8,-1.6)(9.8,2.6)
\psellipse[linewidth=0.04,dimen=outer](5.0,-0.1)(2.0,2.9)

\uput[135](5.0,3.6){$s$}
\uput[-45](10.0,-3.0){$t$}

\rput(9.8,2.8){$\mathcal{E}_x^+ : \; s = |t|^2$}
\rput(2.9,2.9){$\mathcal{E}_x^-: \; \frac{s^2}{a^2}-\frac{2s}{a}+ \frac{|t|^2}{b^2}=0$}

\psline[linewidth=0.04cm,linestyle=dashed](1,0.0)(9,0)

\psline[linewidth=0.06cm,linecolor=red]{->}(5,0.0)(5.7,0)

\psline[linewidth=0.06cm,linecolor=red]{->}(5,0.0)(5,-0.7)

\psdots[dotsize=0.14](1,0)

\psdots[dotsize=0.14](3.02,0)

\psdots[dotsize=0.14](5,0)
\uput[135](5,0){$o_r$}
\psdots[dotsize=0.14](6.98,0)

\psdots[dotsize=0.14](9,0)

\psdots[dotsize=0.14](5,2.8)

\psdots[dotsize=0.14](5,-3)
\put(5,-3.3){$x$}
\end{pspicture} 
\caption{The ellipsoids $\E_x^+$ and $\E_x^-$ in a well-chosen chart} \label{fig_E+_et_E-_dans_carte}
\end{figure}
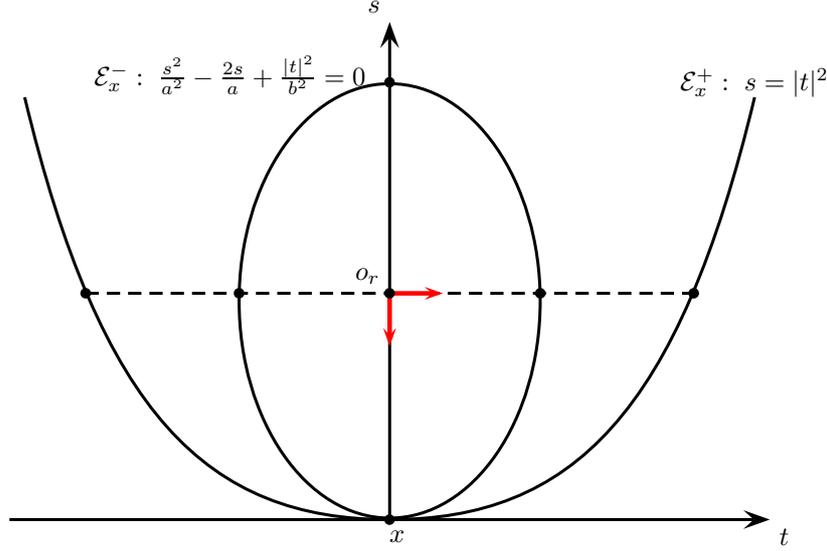
\end{proof}

As a consequence of this lemma, we see that there exists $R\geqslant 0$, depending on $x$ and $\varepsilon$, such that for $r\geqslant R$, we have $f_{x,u}(r)\leqslant \sqrt{\frac{1+\varepsilon}{1-\varepsilon}} + \varepsilon$ for $u\in T_x\dC\smallsetminus\{0\}$ or $u\in\R.ox\smallsetminus\{0\}$. Let us define $R(\varepsilon,x)$ as the smallest $R\geqslant 0$ satisfying this property. Now, the continuity of the functions $x \mapsto \h_x^{\pm}$ (Lemma \ref{lem:continuity}) implies that the function $x \mapsto R(\varepsilon,x)$ is also continuous. Hence, if we set
$$R(\eps):= \sup_{x \in \dC} R(\varepsilon,x), $$
we have that for any $x\in\dC$ and $r\geqslant R(\varepsilon)$, $f_{x,u}(r)\leqslant \sqrt{\frac{1+\varepsilon}{1-\varepsilon}} +\eps$ for $u\in T_x\dC\smallsetminus\{0\}$ or $u\in\R.ox\smallsetminus\{0\}$. Now, each $u\in \R^n$ can be decomposed as $u=u_1+u_2$ with $u_1\in T_x\dC$ and $u_2\in\R.ox$. Remark that $u_1$ and $u_2$ are orthogonal for $\h_x^+$ as well as for $\h_x^-$, so that 
$$\h_x^{\pm}(o_r,u) = \sqrt{\h_x^{\pm}(o_r,u_1)^2 +\h_x^{\pm}(o_r,u_2)^2}.$$ 
For $r>R(\varepsilon)$, we have 

$$
f_{x,u}(r) = \frac{\h^-_x(o_r,u)}{\h^+_x(o_r,u)}  = \sqrt{\frac{\h^-_x(o_r,u_1)^2+\h^-_x(o_r,u_2)^2}{\h^+_x(o_r,u_1)^2+\h^+_x(o_r,u_2)^2}} \leqslant \sqrt{\frac{1+\varepsilon}{1-\varepsilon}} + \eps,\ u\in\R^ n\smallsetminus\{0\}.
$$

That means that for any $x\in\dC$ and $z\in [ox)$ such that $\d(o,z)\geqslant R(\varepsilon)$, we have 
$$1 \leqslant \frac{\F(z,\cdot)}{\h^+_x(z,\cdot)} \leqslant \frac{\h^-_x(z,\cdot)}{\h^+_x(z,\cdot)} \leqslant \sqrt{\frac{1+\varepsilon}{1-\varepsilon}} + \eps.$$
This proves property (3).
\end{proof}

So we get that Hilbert geometries are asymptotically Riemannian:
\begin{cor}\label{cor_suffit}
Let $(\C,\d)$ be a regular Hilbert geometry and $o\in\C$ a base point. For any $C>1$, there exists $R\geqslant 0$ and a continuous Riemannian metric $g$ on $\C\smallsetminus B(o,R)$ such that
$C^{-1} \sqrt{g} \leqslant F_{\C} \leqslant C\sqrt{g}.$
\end{cor}
\begin{proof}
Take the metric $g$ given for $z\in \C\smallsetminus B(o,R)$ by $\sqrt{g_z} = \mathtt{h}_{z^+}$, where $z^+ = [oz)\cap \dC$, and  $\mathtt{h}_{z^+}$ is given by the last lemma.
\end{proof}

We will need the following version of Proposition \ref{prop_hilbert_are_bilipschitz_at_infinity} in section \ref{sec_bottom_of_essential_spectrum}:

\begin{cor}\label{cor_result_for_essential_spectrum}
Let $(\C,\d)$ be a regular Hilbert geometry. To each $x\in\dC$, we can associate a (non-complete) Riemannian hyperbolic metric $\h_x$ defined on an open neighborhood $\mathcal{O}_x$ of $[ox)$ which satisfies the following properties. 
\begin{enumerate}
 \item The application $x\longmapsto \h_x$ is continuous.
 \item We have $\displaystyle \bigcup_{x\in\dC} \mathcal{O}_x = \C$.
 \item The metric $\mathtt{h}_x$ has the same geodesics as $F_{\C}$ on $\mathcal{O}_x$.
 \item Let $C>1$ and 
$$\mathcal{U}_x(C) = \left\{z\in\mathcal{O}_x,\ C^{-1} \leqslant  \frac{F_{\C}(z,\cdot)}{\mathtt{h}_x(z,\cdot)} \leqslant C\right\}.$$ 
There exists $R=R(C)$ such that, for any $x\in\dC$, the intersection of $\mathcal{U}_x(C)$ with $\C\smallsetminus B(o,R)$ is an open neighborhood of $[ox)$ in $\C\smallsetminus B(o,R)$. In particular, we have $\displaystyle \C\smallsetminus B(o,R) \subset \bigcup_{x\in\dC} \mathcal{U}_x(C)$.
\end{enumerate}
\end{cor}
\begin{proof}
We use the objects introduced in the proof of Proposition \ref{prop_hilbert_are_bilipschitz_at_infinity}. We let $\h_x$ be the metric defined by the osculating sphere $S_x$ at $x$ which center is $o$. This is a metric on $\mathcal{O}_x = \E_x \cap \C$, which is an open neighborhood of $[ox)$. It is immediate that $h_x$ and $\mathcal{O}_x$ satisfy the first three points. For the fourth one, pick $\varepsilon>0$ and consider the ellipsoids $\E_x^+$ and $\E_x^-$ which depend on $\varepsilon$. Remark that, as $\E_x^- \subset S_x \subset \E_x^+$ and $\E_x^- \subset \C \subset \E_x^+$, we always have
$$\frac{\h^+_x(z,\cdot)}{\h^-_x(z,\cdot)} \leqslant \frac{F_{\C}(z,\cdot)}{\mathtt{h}_x(z,\cdot)} \leqslant \frac{\h^-_x(z,\cdot)}{\h^+_x(z,\cdot)},$$
for all $z\in \E_x^-$. 

Now, we proved above that there is some $R(\varepsilon)\geqslant 0$ such that, for all $z\in [ox)\smallsetminus B(o,R(\varepsilon))$,  
$$\frac{\h^-_x(z,\cdot)}{\h^+_x(z,\cdot)}\leqslant \sqrt{\frac{1+\varepsilon}{1-\varepsilon}} + \eps.$$

Hence the intersection of the set $\mathcal{U}_x\left(\sqrt{\frac{1+\varepsilon}{1-\varepsilon}} + 2\eps\right)$ with $\C\smallsetminus B(o,R(\varepsilon))$ is an open neighborhood of $[ox)$ in $\C\smallsetminus B(o,R(\varepsilon))$. Since $\varepsilon>0$ is arbitrary, this proves the fourth point.
\end{proof}

\begin{rem}
Note that the metric $\h_x$ in this Corollary is different from the one in in Proposition \ref{prop_hilbert_are_bilipschitz_at_infinity}. In particular, the metric $\h_x$ of the Corollary is independent of $C$.
\end{rem}


\section{Bottom of the spectrum for asymptotically Riemannian metrics} \label{sec_bottom_of_spectrum}

Let $F$ be a $C^2$ Finsler metric on a manifold $M$. Let $\Omega^F$ be the Holmes--Thompson volume for $F$. The \emph{volume entropy} $h$ of $F$ is defined by
\begin{equation*}
 h := \limsup_{R \rightarrow +\infty} \frac{1}{R}\ln \int_{B^F(R)} \Omega^F.
\end{equation*}
In this section, we will show the following

\begin{thm}\label{thm_lambda_1_for_riem_at_infinity}
 Let $F$ be an asymptotically Riemannian $C^2$ Finsler metric on a $n$-manifold $M$. Let $h$ be the volume entropy of $F$ and $\lambda_1$ be the bottom of the spectrum of the Finsler Laplacian $-\Delta^F$. Then,
\begin{equation*}
 \lambda_1 \leqslant h^2/4 .
\end{equation*}
\end{thm}

The idea of the proof of Theorem \ref{thm_lambda_1_for_riem_at_infinity} follows the Riemannian one: we show that we can choose $s$ such that the function $e^{-s d(O,x)}$ has a Rayleigh quotient as close as we want to $h^2/4$. The difficulty is in the control of the Rayleigh quotient. In Section \ref{subsec_control_of_symbol} we show how we can manage to control the Rayleigh quotient by controlling the symbol of the Finsler Laplacian.

As we proved that regular Hilbert geometries are asymptotically Riemannian (Corollary \ref{cor_suffit}), and we know that the volume entropy is $n-1$ (\cite{ColboisVerovic:hilbert_geometry}), we deduce the upper bound in Theorem \ref{thmintro_bottom}:
\begin{cor}
Let $\left(\C, d_{\C} \right)$ be a regular Hilbert geometry. Let $\lambda_1$ be the bottom of the spectrum of the Finsler Laplacian $-\Delta^{\F}$. Then
\begin{equation*}
 \lambda_1 \leqslant \frac{(n-1)^2}{4} \, .
\end{equation*} 
\end{cor}

Note that, for generic  asymptotically Riemannian Finsler metrics, we do not always have $\lambda_1 >0$, even when the volume entropy is positive. Indeed, there exists examples of Riemannian metrics on the universal cover of a manifold such that the volume entropy is positive and $\lambda_1=0$ (for instance the solvmanifold described in \cite{BolsinovTaimanov}). However, in the case of regular Hilbert metrics, this is not possible as the next lemma asserts, which gives the lower bound of Theorem \ref{thmintro_bottom}.

\begin{lem}
 Let $\left(\C, d_{\C} \right)$ be a regular Hilbert geometry. Let $\lambda_1$ be the bottom of the spectrum of the Finsler Laplacian $-\Delta^{\F}$. Then $\lambda_1 >0$.
\end{lem}

\begin{proof}
 By \cite{ColboisVerovic:hilbert_geometry}, we know that a regular Hilbert metric is bi-Lipschitz equivalent to the hyperbolic space, so, by \cite[Theorem 4]{BarthelmeColbois_eigenvalue_control} (that we recall below) and the Min-Max principle, we deduce that the $\lambda_1$ of $F_{\C}$ is bounded below by $C^{-1} (n-1)$, where $C>1$ is a constant depending on $n$ and the bi-Lipschitz control.
\end{proof}

\begin{thm} [Barthelm\'e--Colbois \cite{BarthelmeColbois_eigenvalue_control}]
Let $F$ and $F_0$ be two Finsler metrics on an $n$-manifold $M$.
Suppose that there exists $C >1$ such that, for any $(x,v) \in TM\smallsetminus \{0\}$,
\begin{equation*}
 C^{-1} \leq \frac{F(x,v)}{F_{0}(x,v)} \leq C.
\end{equation*}
 Let $C_1$ and $C_2$ be the quasireversibility constants of $F$ and $F_0$ respectively. Then, there exists a constant $K\geq 1$, depending on $C$, $C_1$, $C_2$ and $n$, such that, for any $f \in H^1(M)$,
\begin{equation*}
 C^{-K} \leq \frac{E^F(f)}{E^{F_{0}}(f)} \leq C^K.
\end{equation*}
\end{thm}
 Note that in \cite{BarthelmeColbois_eigenvalue_control} this Theorem is stated for $M$ compact, but stays true for non-compact manifolds without any change to the proof.

\subsection{Control of the symbol for pointwise bi-Lipschitz metrics} \label{subsec_control_of_symbol}

In this section we prove that, given a bi-Lipschitz control of a Finsler metric by a Riemannian one, we can control the symbol of the Finsler Laplacian by the dual Riemannian metric. Note that this result is not as clear as in Riemannian geometry, as the symbol of the Finsler Laplacian a priori depends on derivatives of the Finsler metric.

\begin{prop} \label{prop_control_of_sigma}
 Let $F$ be a Finsler metric on a $n$-manifold $M$, $x \in M$, and $g_x$ a scalar product on $T_xM$.
Suppose that there exists $C \geqslant 1$ such that, for all $v \in T_xM\smallsetminus\{0\}$,
\begin{equation*}
 C^{-1} \leqslant \frac{F(x,v)}{\sqrt{g_x(v,v)}} \leqslant C.
\end{equation*} 
Then there exists a constant $C'\geqslant 1$, depending only on $C$ and $n$, such that, for all $p \in T^{\ast}_x M$,
\begin{equation*}
 C'^{-1} \leqslant \frac{\sigma_F(p,p)}{g_x^{\ast}(p,p)} \leqslant C'.
\end{equation*}
Furthermore, 
\begin{equation*}
 \lim_{C\rightarrow 1} C'(C,n) = 1.
\end{equation*}
\end{prop}

In \cite{BarthelmeColbois_eigenvalue_control} the first two authors gave a proof of the existence of a $C'$ satisfying the inequality, but not the limit condition. Hence, we here do the proof with a bit more care to ensure this second condition.

Let us fix a $C^2$ Riemannian metric $F_0$ on $M$ such that $F_0(x, \cdot) = \lVert \cdot \rVert_{g_x}$. Let $X_0$ and $X$ be the geodesic vector fields associated with $F_0$ and $F$ respectively. There exists a function $m \colon M \rightarrow (0,+\infty)$ and a vertical vector field $Y \colon HM \rightarrow VHM$ such that $X = m X_0 + Y$. Actually, we have $m=\frac{F_0}{F}$.\\

Before going on to the proof, we start by stating some results that we will need (the proofs are quite elementary and can be found in \cite{BarthelmeColbois_eigenvalue_control}):

\begin{lem} \label{lem_alpha_and_omega}
Let $F$ and $F_0$ be two Finsler metrics on $M$, $X$ and $X_0$ the associated geodesic vector fields. Let $m \colon HM \longrightarrow \R$ be the function $m=\frac{F_0}{F}$ and $\mu \colon M \longrightarrow \R$ be defined by
\begin{equation*}
\mu (x) :=  \left(\voleucl \left(\S^{n-1} \right) \right)^{-1} \int_{H_x^{\ast} M} \left(\frac{F_0^{\ast}}{F^{\ast}} \right)^n \left(\ell_{F_0}^{-1}\right)^{\ast} \alpha^{F_0},\ x\in M. 
\end{equation*}
Then $X = m X_0 + Y$ for some vertical vector field $Y \in VHM$, $\Omega^F = \mu \Omega^{F_0}$ and 
$$\alpha^F_{(x,\xi)} = \frac{1}{\mu(x)} \left(\frac{F_0^{\ast}}{F^{\ast}}(\ell_{F}(\xi)) \right)^n  \left(\ell_{F_0}^{-1} \circ \ell_F \right)^{\ast}  \alpha^{F_0}_{(x,\xi)}.$$

\end{lem}

\begin{lem} \label{lem_control_of_m_mu}
 Let $F$ and $F_0$ be two Finsler metrics on a $n$-manifold $M$. Suppose that for some $x\in M$, there exists $C \geqslant1$ such that, for any $v \in T_xM\smallsetminus\{0\}$,
\begin{equation*}
 C^{-1} \leqslant \frac{F(x,v)}{F_{0}(x,v)} \leqslant C.
\end{equation*}
Then for any $v \in T_xM\smallsetminus\{0\}$, $\xi \in H_xM$, we have
\begin{align}
 C^{-1} &\leqslant \frac{F^{\ast}(x,v)}{F^{\ast}_{0}(x,v)} \leqslant C, \label{eq_dual_control}\\
 C^{-n} &\leqslant \mu(x) \leqslant C^n \label{eq_control_volume} , \\
 C^{-1} &\leqslant m(x,\xi) \leqslant C \label{eq_control_m}.
\end{align}
\end{lem}
Note that the result was stated in \cite{BarthelmeColbois_eigenvalue_control} for a uniform bi-Lipschitz control (that is, $C$ was supposed not to depend on $x\in M$), but the proof stays exactly the same in this case of pointwise bi-Lipschitz control.

\begin{proof}[Proof of Proposition \ref{prop_control_of_sigma}]
Let $p \in T_x^{\ast}M\smallsetminus\{0\}$ be fixed. Let $\lVert \cdot \rVert_{g_x^{\ast}}$ be the norm on $T_x^{\ast}M$ dual to the scalar product $g_x$. We suppose that $\lVert p\rVert_{g_{x}^{\ast}} = 1$. Let $\phi \colon M \rightarrow \R$ be a smooth function such that $\phi(x) = 0$ and $d\phi_x = p$. Then the norm of $p$ for the symbol metric $\sigma^F$ is
\begin{equation*}
 \lVert p\rVert_{\sigma^F}^2 = \frac{n}{\voleucl\left(\S^{n-1} \right)} \int_{H_xM} \left(L_X \pi^{\ast} \phi \right)^2 \alpha^F.
\end{equation*}
Let us write $c_n:= n \left(\voleucl\left(\S^{n-1}\right) \right)^{-1}$.

Let $F_0$ be a $C^2$ Riemannian metric such that $F_0(x, \cdot) = \lVert \cdot \rVert_{g_x}$. Let $X_0$ and $X$ be the geodesic vector fields associated with $F_0$ and $F$ respectively. There exists $m \colon M \rightarrow \R$ and $Y \colon HM \rightarrow VHM$ such that $X = m X_0 + Y$, so, using Lemma \ref{lem_alpha_and_omega} and the change of variable formula, we get
\begin{align*}
 \lVert p\rVert_{\sigma^F}^2 &= c_n \int_{H_xM} m^2 \left(L_{X_0} \pi^{\ast} \phi \right)^2 \left(\ell_{F_0}^{-1} \circ \ell_F \right)^{\ast} \left[ \mu^{-1} \left(\frac{F_0^{\ast}}{F^{\ast}} \circ \ell_{F_0} \right)^n \alpha^{F_0} \right] \\
    &= c_n \int_{H_xM} \left(m\circ \ell_{F}^{-1} \circ \ell_{F_0} \right)^2 \left(L_{X_0} \pi^{\ast} \phi \circ \ell_{F}^{-1} \circ \ell_{F_0} \right)^2  \mu^{-1} \left(\frac{F_0^{\ast}}{F^{\ast}} \circ \ell_{F_0} \right)^n \alpha^{F_0}.
\end{align*}
Now, using Lemma \ref{lem_control_of_m_mu}, we have that
\begin{align*}
 \lVert p\rVert_{\sigma^F}^2 &\leqslant c_n C^{2n +2} \int_{H_xM} \left(L_{X_0} \pi^{\ast} \phi \circ \ell_{F}^{-1} \circ \ell_{F_0} \right)^2 \alpha^{F_0},\\
 \lVert p\rVert_{\sigma^F}^2 &\geqslant c_n C^{-2n -2} \int_{H_xM} \left(L_{X_0} \pi^{\ast} \phi \circ \ell_{F}^{-1} \circ \ell_{F_0} \right)^2 \alpha^{F_0}.
\end{align*}
That means we have
\begin{align*}
\frac{\lVert p\rVert_{\sigma^F}^2}{ \lVert p\rVert_{\sigma^{F_0}}^2} &\leqslant 
  C^{2n +2}\frac{\displaystyle \int_{S^{F_0}_xM} p(\L_{F}^{-1} \circ \L_{F_0} (u))^2 \alpha^{F_0}(u)}{\displaystyle\int_{S^{F_0}_xM} p(u)^2 \alpha^{F_0}(u)}.
\end{align*}
But, by continuity of $\L$ (Lemma \ref{lem_continuity_l}), we have $\L_{F}^{-1} \circ \L_{F_0} (u) = u + \varepsilon(u)$ with $F_0(\varepsilon(u))\leqslant \varepsilon(C)$, where $C\longmapsto \varepsilon(C)$ is a continuous function such that $\varepsilon(0)=0$. This gives $$p(\L_{F}^{-1} \circ \L_{F_0} (u))^2=(p(u) + p(\varepsilon(u)))^2 \leqslant p(u)^2 + p(\varepsilon(u))^2 + |2p(u)p(\varepsilon(u))|$$
and $|p(\varepsilon(u))|\leqslant \varepsilon(C)\|p\|_{\sigma^{F_0}}$. So we get, using Cauchy-Schwarz inequality,

\begin{align*}
\frac{\lVert p\rVert_{\sigma^F}^2}{ \lVert p\rVert_{\sigma^{F_0}}^2} &\leqslant 
 \frac{C^{2n +2}}{\lVert p\rVert_{\sigma^{F_0}}^2}\left((1+\varepsilon(C)) \lVert p\rVert_{\sigma^{F_0}}^2 + 2\left(\int_{S^{F_0}_xM} p(u)^2 \alpha^{F_0}(u)\right)^{1/2}\left(\int_{S^{F_0}_xM} p(\varepsilon(u))^2 \alpha^{F_0}(u)\right)^{1/2}\right)\\\\
&\leqslant  C^{2n +2} (1+2\varepsilon(C)).
\end{align*}
The same computations also gives the lower bound.
\end{proof}

\subsection{$\lambda_1$ and volume entropy} \label{subsec_lambda_and_entropy}

We prove here Theorem \ref{thm_lambda_1_for_riem_at_infinity}. Let $o$ be a base point on $M$. For any $x \in M$, define $\rho(x) := d(o,x)$, with $d$ the Finslerian distance.

\begin{claim}
 For any $s\in \R$ such that $2s > h$, we have $e^{-s\rho(\cdot)} \in L^2 \left( M,\Omega^F\right)$.
\end{claim}

\begin{proof}
This fact is straightforward, just using the definition of the volume entropy.
\end{proof}

Choose $C>1$. As $F$ is asymptotically Riemannian, there exists a compact subset $K_C$ of $M$ and, for any $x\in M\smallsetminus K_C$, a scalar product $g_x$ on $T_xM$ such that, for any $v\in T_xM\smallsetminus\{0\}$,
\begin{equation*}
 C^{-1} \leqslant \frac{F(x,v)}{\sqrt{g_x(v,v)}} \leqslant C.
\end{equation*} 

Now let $R(C)>0$ such that the Finslerian metric ball $B^F(o,R(C)) \subset M$, of center $o$ and radius $R(C)$, contains $K_C$. Set 
\begin{equation*}
 f_C(x) := \begin{cases}
            e^{-s R(C)} \quad \text{if } x \in B^F(o,R(C)) \\
	    e^{-s \rho(x)} \quad \text{if } x \in M \smallsetminus B^F(o,R(C)).
           \end{cases}
\end{equation*}
We will start by giving an upper bound on the energy of $f_C$. Let $\lVert \cdot \rVert_{\sigma^F}$ be the norm given by the symbol of $F$. We have
\begin{equation*}
 E^F(f_C) = \frac{n}{\voleucl(\S^{n-1})} \int_{HM} \left(L_X \pi^{\ast} f \right)^2 \ada = \int_M \lVert df \rVert_{\sigma^F}^2 \Omega^F.
\end{equation*}
Hence, if we set $U_C:= M \smallsetminus B^F(o,R(C))$
\begin{equation*}
 E^F(f_C) = \int_{U_C} s^2\lVert d\rho_x \rVert_{\sigma^F_x}^2 e^{-2s \rho(x)} \Omega^F(x).
\end{equation*}

Now, by Proposition \ref{prop_control_of_sigma}, there exists $C'\geqslant 1$ such that, for any $x \in U_C$,
\begin{equation*}
 \lVert d\rho_x \rVert_{\sigma^F_x} \leqslant C' \lVert d\rho_x \rVert_{g_x^{\ast}}  \leqslant C C' \lVert d\rho_x \rVert_{F^{\ast}_x},
\end{equation*}
where the last inequality holds because a $C$-bi-Lipschitz control of two Finsler metrics implies a $C$-bi-Lipschitz control of their dual metrics (see for instance \cite{BarthelmeColbois_eigenvalue_control}).

By definition,
\begin{equation*}
 \lVert d\rho_x \rVert_{F^{\ast},x} = \sup \lbrace d\rho_x(v) \mid v \in T_xM, F(x,v)=1 \rbrace = 1,
\end{equation*}
because $\rho$ is the distance function of $F$.

So we have obtained that 
\begin{equation*}
  E^F(f_C) \leqslant s^2 C^2 C'^2 \int_{U_C} e^{-2s \rho(x)}\Omega^F(x).
\end{equation*}
We also have that 
\begin{equation*}
 \int_{M} f_C(x)^2 \Omega^F(x) = \int_{B^F(o,R(C))}  e^{-2s R(C)}\Omega^F(x) +  \int_{U_C} e^{-2s \rho(x)}\Omega^F(x) \geqslant  \int_{U_C} e^{-2s \rho(x)} \Omega^F(x).
\end{equation*}
Therefore,
\begin{equation*}
 R^F(f_C) = \frac{E^F(f_C)}{ \int_{M} f_C(x)^2\Omega^F(x)} \leqslant \frac{s^2 C^2 C'^2 \int_{U_C} e^{-2s \rho(x)}\Omega^F(x)}{\int_{U_C} e^{-2s \rho(x)}\Omega^F(x)} = s^2 C^2 C'^2.
\end{equation*}
This is true for any $s > h/2$ and any $C >1$. Since $\lim_{C\rightarrow 1} C' = 1$, we get
\begin{equation*}
 \lambda_1 = \inf_{f \in L^2(M, \Omega^F)} R^F(f) \leqslant \frac{h^2}{4}\, .\qedhere
\end{equation*}

This finishes the proof of Theorem \ref{thm_lambda_1_for_riem_at_infinity}.

\subsection{Dirichlet spectrum}

By a slight modification of the above proof, we can show that the same bound holds for the first Dirichlet eigenvalue of an asymptotically Riemannian manifold $M$ from which we removed a compact set $K$, \emph{provided} that we know that the function $e^{-h \rho(x)/2}$ is \emph{not} in $L^2(M)$. For a general (asymptotically) Riemannian manifold, this is probably not true. But it is true for example on the universal cover of a compact negatively curved Riemannian manifold: in this case, Margulis \cite{margulis, margulisbook} proved that, when $R$ goes to infinity, the area of the sphere of radius $R$ is equivalent to $C e^{hR}$, for some constant $C>0$, which allows to conclude. We will see below that this argument also applies to regular Hilbert geometries.

Recall that if $K$ is a compact sub-manifold of $M$ of the same dimension, the Dirichlet spectrum on $M\smallsetminus K$ is the spectrum of the operator $-\Delta^F$ seen on the space obtained by completion of $C^{\infty}_{0}(M\smallsetminus K)$, the space of smooth functions with compact support in $M\smallsetminus K$, under the norm given by the sum of the $L2$-norm and the energy. The first eigenvalue can still be obtained via the infimum of the Rayleigh quotient.

\begin{cor} \label{cor_bound_for_dirichlet_eigenvalue}
 Let $(M, F)$ be an asymptotically Riemannian manifold and $K$ a compact sub-manifold of $M$ of the same dimension. Let $\lambda_1(M\smallsetminus K)$ be the bottom of the Dirichlet spectrum of $-\Delta^F$ on $M\smallsetminus K$. Let $o \in M$ and $\rho(x):= d(o,x)$. If the function $x\longmapsto e^{-h \rho(x)/2}$ is not in $L^2(M)$, then
\begin{equation*}
 \lambda_1(M\smallsetminus K) \leqslant \frac{h^2}{4} \, .
\end{equation*}
\end{cor}

\begin{proof}
We use the same notations as above: Let $C >1$, and $R(C)$ be such that, outside of $B(o,R(C))$, $F$ is $C$-bi-Lipschitz to a Riemannian metric. By choosing a larger $R(C)$ if necessary, we can assume that $K \subset B(o,R(C))$.
Now, we just need to modify a tiny bit our test function $f_C$ from above so that it is zero on $\partial K$, and show that the Rayleigh quotient is still as close to $h^2/4$ as we want.\\
Let $f_C$ be a function such that
\begin{equation*}
 f_C(x) := \begin{cases}
            0 \quad \text{if } x \in \partial K\\
	    e^{-s \rho(x)} \quad \text{if } x \in M \smallsetminus B^F(o,R(C)),
           \end{cases}
\end{equation*}
and, furthermore,
\[
 \int_{B^F(o,R(C)) \smallsetminus K} \lVert df_C \rVert_{\sigma^F}^2\ \Omega^F \leqslant 1.
\]
Such a function exists if $R(C)$ is large enough.

Hence, if we set again $U_C:= M \smallsetminus B^F(o,R(C))$, we obtain as above that 
\begin{equation*}
 E^F(f_C) \leqslant 1 + s^2 C^2 C'^2 \int_{U_C} e^{-2s \rho(x)}\ \Omega^F(x).
\end{equation*}
Thus,
\begin{equation*}
 R^F(f_C) \leqslant \frac{1}{\int_M e^{-2s\rho(x)}\ \Omega^F(x)} + s^2 C^2 C'^2
\end{equation*}
Now, as $x\longmapsto e^{-h\rho(x)/2}$ is not in $L^2(M)$, $2s$ can be taken close enough to $h$, so that $\int_M e^{-2s\rho(x)} \Omega^F(x)$ arbitrarily large. Finally, as $C$ can be taken arbitrarily close to $1$ and $\lim_{C \rightarrow 1} C'(C) =1$, we obtain that
\[
 \inf R^F(f_C) \leq \frac{h^2}{4} \, , 
\]
which ends the proof.
\end{proof}

Using this, we can now prove the corresponding result about regular Hilbert geometries, which will be useful to compute the bottom of the essential spectrum in the next section. 

\begin{cor} \label{cor_dirichlet_spectrum_for_hilbert}
 Let $\left(\C, d_{\C} \right)$ be a regular Hilbert geometry and $K$ be a compact subset of $\C$ with smooth boundary. Let $\lambda_1(\C \smallsetminus K)$ be the bottom of the Dirichlet spectrum of $-\Delta^{F_{\C}}$ on $\C \smallsetminus K$. Then
\begin{equation*}
 \lambda_1(\C \smallsetminus K) \leqslant \frac{(n-1)^2}{4} \, .
\end{equation*} 
\end{cor}

\begin{proof}
Let $o \in \C$ and $\rho(x) := \d(o,x)$. Thanks to corollary \ref{cor_bound_for_dirichlet_eigenvalue}, we only have to show that the function $x\longmapsto e^{-(n-1)\rho(x)/2}$ is not in $L^2(\C, \Omega^{F_{\C}})$.\\
In \cite{ColboisVerovic:hilbert_geometry}, the second and fourth authors gave a precise evaluation of the volume form of a regular Hilbert geometry. Their computations imply in particular that there exists some constant $C >0$ such that, for any measurable function $f:[0,+\infty)\longrightarrow \R$,
\[
 \int f\circ\rho\ \Omega^{F_C} \geqslant C \int_0^{+\infty}  f(r)e^{(n-1)R}\ dr.
\]
(See the proof of Theorem 3.1 in \cite{ColboisVerovic:hilbert_geometry}. The computations are done for the Busemann--Hausdorff volume, but the ratio between Busemann--Hausdorff and Holmes--Thompson volumes is uniformly bounded, with bounds depending only on the dimension (see for instance \cite{BuragoBuragoIvanov}), so their result applies.)\\
The conclusion is immediate:
\[
 \int e^{-(n-1)\rho(x)}\Omega^{F_C}(x) \geqslant C  \int_0^{+\infty} dr = +\infty \qedhere
\]
\end{proof}

\section{Bottom of the essential spectrum} \label{sec_bottom_of_essential_spectrum}

Coming back to regular Hilbert geometries, we will now study the essential spectrum and prove Theorem \ref{thmintro_essential_bottom}.
\begin{thm}\label{thm_essential_spectrum}
 Let $\left(\C, d_{\C} \right)$ be a regular Hilbert geometry. Let $\sigma_{\textrm{ess}}(F_{\C})$ be the essential spectrum of $-\Delta^{F_{\C}}$. Then
\begin{equation*}
 \inf \sigma_{\textrm{ess}}(F_{\C}) = \frac{(n-1)^2}{4} \,.
\end{equation*}
\end{thm}

So, if the $\lambda_1$ of a regular Hilbert geometry is strictly less than $(n-1)^2/4$, then it is a true eigenvalue, contrarily to the hyperbolic case where the $\lambda_1$ is just the infimum of the spectrum.

Note that, in the next section, we will construct examples of Hilbert geometries with eigenvalue strictly smaller than $(n-1)^2/4$. Indeed, we will construct examples with arbitrarily many, arbitrarily small eigenvalues.

To prove our result on the essential spectrum, we will use the Cheeger inequality for weighted Laplacians and control the Cheeger constant in regular Hilbert geometries using Corollaries \ref{cor_suffit} and \ref{cor_result_for_essential_spectrum}.

\subsection{Cheeger constant, weighted Laplacians and essential spectrum}

If $F$ is a Finsler metric on a manifold $M$, then (see \cite{moi:natural_finsler_laplace}) $\Delta^F$ is a weighted Laplacian with symbol $\sigma^F$ and symmetric with respect to the volume $\Omega^F$. Hence, we have the following lower bound for the first eigenvalue of $-\Delta^F$:
\begin{prop}[Cheeger Inequality] \label{prop_Cheeger_inequality}
 Let $M$ be a non-compact manifold and $F$ a Finsler metric on $M$. Let $d\vol^{\sigma^F}$ be the volume form of the Riemannian metric dual to $\sigma^F$, $d\mathrm{area}^{\sigma^F}$ the associated area element and $\mu \colon M \rightarrow \R$ the function such that $\Omega^F = \mu d\mathrm{vol}^{\sigma^F}$. Set 
\begin{equation*}
 \cheegerapoids(M) := \inf_{D} \left\{ \frac{\int_{\partial D} \mu(x) d\mathrm{area}^{\sigma^F} }{\int_{D} d\mathrm{vol}^{\sigma^F} } \right\},
\end{equation*}
where the infimum is taken over all compact domains $D$ with smooth boundary.

If $\lambda_1$ is the bottom of the spectrum of $-\Delta^F$ on $M$, then 
\begin{equation*}
 4 \lambda_1 \geqslant \cheegerapoids(M)^2.
\end{equation*}
\end{prop}

We do not provide the proof as it is the exact same as for the traditional Cheeger inequality (see for instance \cite{SchoenYau:lectures}). To study the essential spectrum, we also need the decomposition principle of Donnelly and Li, which states that the essential spectrum is independent of the behavior of the operator on any compact subset:
\begin{prop}[Decomposition Principle of Donnelly and Li \cite{DonnellyLi}] \label{prop_decomposition_principle}
 Let $M$ be a non-compact manifold and $F$ a Finsler metric on $M$. Let $M'$ be a compact sub-manifold of $M$ of same dimension. Then
\begin{equation*}
  \sigma_{\textrm{ess}}(M,F) = \sigma_{\textrm{ess}}(M \smallsetminus M',F).
\end{equation*}
In particular,
\begin{equation*}
 \cheegerapoids(M \smallsetminus M')^2 \leqslant 4 \inf \sigma_{\textrm{ess}}(F).
\end{equation*}
\end{prop}

We also have the following known result. As we did not find any reference, we provide a proof.
\begin{lem}
Let $\{M'_i\}$ be an increasing family of compact sub-manifolds of $M$ of the same dimension, such that $\cup_i M'_i = M$. Then
$$\inf \sigma_{\textrm{ess}}(M,F) = \lim_{i\to \infty} \l_1(M \smallsetminus M'_i,F),$$
where $\l_1(M \smallsetminus M'_i,F)$ denotes the Dirichlet spectrum of $M \smallsetminus M'_i$.
\end{lem}
\begin{proof}
Let us write $\lambda_1^i := \l_1(M \smallsetminus M'_i,F) $.
By the Decomposition principle, we have that, for all $i$, $\lambda_1^i \leqslant \inf \sigma_{\textrm{ess}}(M,F)$. We suppose that $\lambda_1^i < \inf \sigma_{\textrm{ess}}(M,F)$, otherwise we are done. Let $\lambda = \lim_{i\to \infty} \lambda_1^i $, which exists because, as $\{M'_i\}$ is increasing, the sequence $\{\lambda_1^i\}$ is nondecreasing. To prove that $\lambda$ is in the essential spectrum, we are going to show that, for any $\eps>0$, there exists a family of functions $f_i \in L^2(M)$, with disjoint supports, such that
$$
\lVert -\Delta^F f_i - \lambda f_i \rVert \leqslant \eps \lVert f_i \rVert,
$$
where $\lVert \cdot \rVert$ denotes the $L^2$-norm with respect to $\Omega^F$.

Let $\eps>0$. As $\lambda_1^i$ is an eigenvalue with finite multiplicity of $-\Delta^F$ on $M \smallsetminus M'_i$, we can find a function $f_i \in L^2(M \smallsetminus M'_i)$ with compact support such that
$$
\lVert -\Delta^F f_i - \lambda_1^i f_i \rVert \leqslant \eps \lVert f_i \rVert.
$$
Up to taking a subsequence, we can suppose that $\text{supp} f_i \subset M'_{i+1}$, so that $\text{supp} f_i \subset M'_{i+1} \smallsetminus M'_i$. Hence, for any $i \neq j$, we have $\text{supp} f_i \cap \text{supp} f_j = \emptyset$. So, for $i$ large enough,
\begin{equation*}
 \lVert -\Delta^F f_i - \lambda f_i \rVert \leqslant \lVert -\Delta^F f_i - \lambda_1^i f_i \rVert  + |\lambda - \lambda_1^i| \lVert f_i \rVert \leqslant 2\eps \lVert f_i \rVert. \qedhere
\end{equation*}
\end{proof}

This gives a part of Theorem \ref{thm_essential_spectrum}.

\begin{cor}
Let $\left(\C, d_{\C} \right)$ be a regular Hilbert geometry. Then
$$\inf \sigma_{\textrm{ess}}(F_{\C}) \leqslant (n-1)^2/4.$$
\end{cor}
\begin{proof}
Pick $o\in\C$. Then Corollary \ref{cor_dirichlet_spectrum_for_hilbert} gives that, for any $i\geqslant 1$, $\l_1(\C \smallsetminus \overline{B(o,i)}) \leqslant (n-1)^2/4$. The previous lemma allows us to conclude.
\end{proof}

\subsection{Essential spectrum of regular Hilbert geometries}

The next few lemmas will allow us to prove the inequality $\inf \sigma_{\textrm{ess}}(F_{\C}) \geqslant (n-1)^2/4$ and thus conclude the proof of Theorem \ref{thm_essential_spectrum}.
Denote by $\sigma$ the symbol of $-\Delta^{F_{\C}}$, by $\cheegerapoids$ the weighted Cheeger constant associated with $\sigma$ and $\Omega^{F_{\C}}$ and by $\cheeger$ the traditional Cheeger constant for the Riemannian metric dual to $\sigma$.

Let $o \in \C$ be fixed and $K$ a relatively compact open subset of $\C$.

\begin{lem}\label{lem_comparison_between_F_and_sigma}
 For any $C>1$, there exists a constant $R = R(C) >0$ and a constant $C_1 = C_1(C) \geqslant1$ such that, on $\C \smallsetminus B(o,R)$, we have:
\begin{equation*}
 C_1^{-1} \leqslant \frac{\sigma^{\ast}}{F^2} \leqslant C_1.
\end{equation*}
Furthermore, $C_1$ tends to $1$ as $C$ tends to $1$.
\end{lem}

\begin{proof}
Let $C >1$. According to Corollary \ref{cor_suffit}, there exists $R= R(C)>0$ and a Riemannian metric $g$ on $\C\smallsetminus B(o,R)$ such that $C^{-1} g \leqslant F^2_{\C} \leqslant C g$. By Proposition \ref{prop_control_of_sigma}, there exists a constant $C'=C'(C,n)>1$ such that $(C'C)^{-1} F^2_{\C}  \leqslant \sigma^{\ast} \leqslant C'C F^2_{\C}$ on all of $\C \smallsetminus B(o,R) $. Finally, still according to Proposition \ref{prop_control_of_sigma}, $C'C$ tends to $1$ when $C$ tends to $1$, so we can set $C_1=C'C$.
\end{proof}

\begin{lem}\label{lem_cheeger_a_poid_vs_cheeger}
 For any $C>1$, there exists a constant $R = R(C) >0$ and a constant $C_2 = C_2(C) >1$ such that 
\begin{equation*}
 \cheegerapoids \left( \C \smallsetminus B(o,R) \right) \geqslant C_2^{-1}\cheeger \left( \C \smallsetminus B(o,R) \right).
\end{equation*}
Furthermore, $C_2$ tends to $1$ as $C$ tends to $1$.
\end{lem}

\begin{proof}
 Let $C >1$. By Lemma \ref{lem_comparison_between_F_and_sigma}, there exist constants $R= R(C)>0$ and $C_1\geqslant 1$ such that,  on $\C \smallsetminus B(o,R) $, we have
$C_1^{-1} F^2_{\C}  \leqslant \sigma^{\ast} \leqslant C_1 F^2_{\C}$. Let $\mu \colon M \rightarrow \R$ be the function such that $\Omega^F = \mu d\vol^{\sigma^F}$. By Lemma \ref{lem_control_of_m_mu}, we have $C_1^{-n} \leqslant \mu(x) \leqslant C_1^n$ for any $x \in\C \smallsetminus B(o,R)$. So we get that, for any compact domain $D$ in $\C \smallsetminus B(o,R)$ with smooth boundary,
\begin{align*}
 C_1^{-n} \int_{\partial D} d\mathrm{area}^{\sigma} &\leqslant \int_{\partial D} \mu(x) d\mathrm{area}^{\sigma^F} \leqslant C_1^n \int_{\partial D} \mu(x) d\mathrm{area}^{\sigma^F} \\
 C_1^{-n} \int_{D} d\mathrm{vol}^{\sigma} &\leqslant \int_{D} \Omega^{F_{\C}} \leqslant C_1^n \int_{D} d\mathrm{vol}^{\sigma}.
\end{align*}
Therefore, setting $C_2 = C_1 ^{2n}$ gives the claim.
\end{proof}

\begin{lem}\label{lem_cheeger_vs_hyperbolic}
 For any $C>1$, there exists a constant $R = R(C) >0$ and a constant $C_3 = C_3(C) \geqslant 1$ such that 
\begin{equation*}
 \cheeger \left( \C \smallsetminus B(o,R) \right) \geqslant C_3^{-1} (n-1).
\end{equation*}
Furthermore, $C_3$ tends to $1$ as $C$ tends to $1$.
\end{lem}

\begin{proof}
 Let $C>1$. Let $R= R(C)\geqslant0$, $\mathcal{U}_x(C)$ and $\h_x$, $x\in \dC$, be given by Corollary \ref{cor_result_for_essential_spectrum}. Let $D$ be a compact domain in $\C \smallsetminus B(o,R)$ with smooth boundary. As the goal is to control the Cheeger constant, we can suppose that $D$ is a convex domain, because convex sets minimize the ratio of area over volume.

For each $x\in \dC$, let $K_x$ be a family of open cones with vertex $o$ such that, for any $x \in \dC$, $K_x\cap D \subset \mathcal{U}_{x}(C)$ and $ \displaystyle \cup_{x \in \dC} K_x$ covers $D$. Such a family exists because $\cup_{x \in \dC} \mathcal{U}_x $ openly covers $\C \smallsetminus B(o,R)$. Remark that the boundary of $K_x$ is a union of geodesics of $F_{\C}$, which are also geodesics of $\h_{x}$.

Now, by compactness of $D$, there exist $x_1, \dots, x_k \in \partial \C$ such that $\cup_i K_{x_i}$ openly covers $D$. By choosing the cones $K_{x_i}$ to be smaller if necessary, we can assume that the domain $D$ is partitioned into $\cup_{1 \leqslant i \leqslant k} (K_{x_i} \cap D)$.

\begin{claim} \label{claim_computation_in_hyperbolic}
 For any $1 \leqslant i \leqslant k$, we have 
\begin{equation*}
 \frac{\mathrm{Area}^{\h_{x_{i}}} \left(K_{x_i} \cap \partial D \right) }{ \mathrm{Vol}^{\h_{x_{i}}} \left(K_{x_i} \cap D \right) } \geqslant (n-1).
\end{equation*}
\end{claim}

\begin{proof}[Proof of Claim \ref{claim_computation_in_hyperbolic}]
As $\h_{x_{i}}$ is a hyperbolic metric and the sides of $K_{x_i}$ are geodesics of $\h_{x_{i}}$, we have that
\begin{equation*}
 \frac{\mathrm{Area}^{\h_{x_{i}}} \left(K_{x_i} \cap \partial D \right) }{ \mathrm{Vol}^{\h_{x_{i}}} \left(K_{x_i}\right) } \geqslant (n-1).
\end{equation*}
Indeed, as we are in the hyperbolic setting, this can be proved by a direct computation using the divergence formula. As $\mathrm{Vol}^{\h_{x_{i}}} \left(K_{x_i}\right) \geqslant \mathrm{Vol}^{\h_{x_{i}}} \left(K_{x_i} \cap D \right)$, we get the claim.
\end{proof}

For all $1\leqslant i \leqslant k$ holds $ C^{-1} \h_{x_{i}} \leqslant \F^2 \leqslant  C \h_{x_{i}}$ on $\mathcal{U}_{x_{i}}(C)$. As in Lemma \ref{lem_comparison_between_F_and_sigma}, there exists a constant $C_1':=C'_1(C)$ such that, on $\mathcal{U}_{x_{i}}(C)$,
$$C_1'^{-1} \h_{x_{i}} \leqslant \sigma^{\ast} \leqslant  C_1' \h_{x_{i}},\ 1\leqslant i \leqslant k,$$  and, furthermore, $\lim_{C\to 1} C'_1(C) =1$.

Hence, for any domain $U$ in $\mathcal{U}_{x_{i}}(C)$, and in particular for $K_{x_i} \cap D $, we have 
\begin{align*}
 C_1'^{-n-1} \int_{\partial U} d\mathrm{area}^{\h_{x_{i}}} &\leqslant \int_{\partial U} d\mathrm{area}^{\sigma} \leqslant C_1'^{n-1} \int_{\partial U} d\mathrm{area}^{\h_{x_{i}}}, \\
 C_1'^{-n} \int_{U} d\mathrm{vol}^{\h_{x_{i}}} &\leqslant \int_{U} d\mathrm{vol}^{\sigma} \leqslant C_1'^n \int_{U} d\mathrm{vol}^{\h_{x_{i}}}.
\end{align*}
So, thanks to Claim \ref{claim_computation_in_hyperbolic}, we get 
\begin{equation*}
 \frac{\mathrm{Area}^{\sigma} \left(K_{x_i} \cap \partial D \right) }{ \mathrm{Vol}^{\sigma} \left(K_{x_i} \cap D \right) } \geqslant C_1'^{-2n-1} (n-1).
\end{equation*}
Setting $C_3 :=  C_1'^{2n+1}$, we have
\begin{equation*}
 \mathrm{Area}^{\sigma}(\partial D) = \sum_{i=1}^k \mathrm{Area}^{\sigma} (K_{x_i} \cap \partial D) 
    \geqslant C_3^{-1} (n-1) \sum_{i=1}^k \mathrm{Vol}^{\sigma} (K_{x_i} \cap D)  = C_3^{-1} (n-1) \mathrm{Vol}^{\sigma}(D),
\end{equation*}
Finally, $C_3$ tends to $1$ when $C$ tends to $1$, because it is the case for $C_1'$.
\end{proof}

We can now complete a
\begin{proof}[Proof of Theorem \ref{thm_essential_spectrum}] It remains to show that the infimum of the essential spectrum is greater than $(n-1)^2/4$. Let $C>1$. Combining Proposition \ref{prop_decomposition_principle} with Lemmas \ref{lem_cheeger_a_poid_vs_cheeger} and \ref{lem_cheeger_vs_hyperbolic}, we see that 
$$ 4 \inf \sigma_{\textrm{ess}}(F) \geqslant (C_2(C)C_3(C))^{-2}(n-1)^2$$
for some $C_2(C),C_3(C)\geqslant 1$. When $C$ goes to $1$, $C_2(C)$ and $C_3(C)$ also tend to $1$, hence
\begin{equation*}
4 \inf \sigma_{\textrm{ess}}(F) \geqslant (n-1)^2. \qedhere 
\end{equation*}
\end{proof}

\section{Small eigenvalues} \label{sec_small_eigenvalues}

The Hilbert geometries of simplices is a very simple one:

\begin{prop}[\cite{Delaharpe:hilbert_metric_simplices}]\label{dlh}
The Hilbert geometry defined by a simplex of $\RP^n$ is isometric to a normed vector space of dimension $n$.
\end{prop}

In this section, we construct, by taking $C^2$-approximations of simplices, properly convex sets with arbitrarily many, arbitrarily small eigenvalues.

\begin{thm} \label{thm_small_eigenvalues}
 Let $\eps >0$ and $N \in \N$. There exists a regular Hilbert geometry $(\C,\d)$ such that the $N$ first eigenvalues of $-\Delta^{F_{\C}}$ are below $\eps$.
\end{thm}

Let $S_m$ be a family of simplices in $\R^{n}$ converging in the Hausdorff distance to a simplex $S_{\infty}$ and such that, for all $m$, $\overline{S_{\infty}}\subset S_m$. Let $\C_m$ be a family of convex subsets of $\R^n$ defining regular Hilbert geometries such that, for all $m$, $S_{\infty}\subset \C_m \subset S_m$.
For simplicity, we write $F_m$ instead of $F_{\C_m}$ and $F_{\infty}$ instead of $F_{S_{\infty}}$. We write $\Omega^{F_{\infty}}$ for the Holmes--Thompson volume of $F_{\infty}$.
\begin{lem} \label{lem_uniform_convergence_of_convex}
Let $K$ be a compact set in $S_{\infty}$. Let $\mu_m \colon \S_{\infty} \rightarrow \R^+$ be the function such that $\Omega^{F_m} = \mu_m \Omega^{F_{\infty}}$. For any $\eta>1$, there exists $M=M(K,\eta) \in \N$ such that, for any $x,y \in K$ and $m \geqslant M$, we have
\begin{equation*}
 \eta^{-1} d_{S_{\infty}}(x,y) \leqslant d_{\C_m}(x,y) \leqslant d_{S_{\infty}}(x,y)
\end{equation*}
and
\begin{equation*}
 \eta^{-1} \leqslant \mu_m(x) \leqslant \eta.
\end{equation*}
\end{lem}

\begin{proof}
As $S_{\infty}\subset \C_m \subset S_m$, we have, for any $x \in S_{\infty}$ and $m\in \N$,
\begin{equation*}
F_{S_{m}}(x,\cdot) \leqslant F_{\C_m}(x,\cdot) \leqslant F_{S_{\infty}}(x,\cdot).
\end{equation*}
As $S_m$ converges to $S_{\infty}$ when $m$ tends to infinity, the ratio $F_{S_{m}}/F_{S_{\infty}}$, defined on $HS_{\infty}$, converges uniformly on compact subsets of $HS_{\infty}$ to $1$. This is enough to conclude.
\end{proof}

Denote by $B_m(x,R)$ the open metric ball of radius $R>0$ centered at $x$ for $\C_m$, $0\leqslant m \leqslant \infty$. For given $R>0$, $m\in \N$ and $x\in S_{\infty}$, we define the function $f_{R,m,x} \colon \C_m \rightarrow \R$ by
\begin{equation*}
 f_{R,m,x}(y) = \begin{cases}
               1 & \text{if } y \in B_m(x,R) \\
	       (R+1) - d_{\C_m}(x,y) & \text{if } y \in B_m(x,R+1) \smallsetminus B_m(x,R) \\
	       0  & \text{if } y \in \C_m \smallsetminus B_m(x,R+1).
              \end{cases}
\end{equation*}

\begin{lem} \label{lem_function_with_small_Rayleigh_quotient}
 Let $\eps>0$. Let $R >0$ be chosen such that $\left((R+1)^{n} - R^n \right)/R^n < \eps/(8n)$. Let $x\in S_{\infty}$ and $K$ be a compact set in $S_{\infty}$ containing $B_m(x,R+1)$ for all $m$ big enough. There exists $M \in \N$, depending on $K$ and $\eps$, such that, for all $m \geqslant M$,
\[
 R^{F_{m}}(f_{R,m,x}) \leqslant \eps/2.
\]
\end{lem}

\begin{proof}
We write $X_m$ for the generator of the geodesic flow of $F_m$ on $H\C_m$. Let us start by giving a first bound on $ R^{F_{m}}(f_{R,m,x})$:
\begin{align*}
 R^{F_{m}}(f_{R,m,x}) &= \frac{n}{\voleucl(\S^{n-1})} \frac{\int_{H\C_m} \left(L_{X_m}\pi^{\ast} f_{R,m,x} \right)^2 A^{F_m}\wedge \left(dA^{F_m}\right)^{n-1} } {\int_{\C_m} f^2 \Omega^{F_m}} \\
     &= \frac{n}{\voleucl(\S^{n-1})}\frac{\int_{H\left(B_m(x,R+1) \smallsetminus \overline{B_m(x,R)} \right)} \left(L_{X_m}\pi^{\ast} f_{R,m,x} \right)^2 A^{F_m}\wedge \left(dA^{F_m}\right)^{n-1} } {\int_{B_m(x,R+1)} f^2 \Omega^{F_m}}\\
     &\leqslant n \frac{\int_{B_m(x,R+1) \smallsetminus B_m(x,R)} \Omega^{F_m}} {\int_{B_m(x,R)} \Omega^{F_m}}.
\end{align*}
So all we have to do now is give an upper bound of $\int_{B_m(x,R+1) \smallsetminus B_m(x,R)} \Omega^{F_m}$ and a lower bound of $\int_{B_m(x,R)} \Omega^{F_m}$ such that their ratio is as small as we want for $R$ and $m$ big.

Let $\eta>1$ be such that 
\begin{align*}
 n \eta^{n+2}(1 - \eta^{n}) &\leqslant \eps/4 \\
 \eta^{n+2} &\leqslant 2.
\end{align*}
By Lemma \ref{lem_uniform_convergence_of_convex}, there exists $M \in \N$, depending on $K$ and $\eps$, such that, for all $m \geqslant M$,
\begin{equation*}
 \eta^{-1} d_{S_{\infty}}(x,y) \leqslant d_{\C_m}(x,y) \leqslant d_{S_{\infty}}(x,y).
\end{equation*}
Hence, for $m \geqslant M$, we have $B_{\infty}(x,R) \subset B_{m}(x,R) \subset B_{\infty}(x, \eta R)$, and 
\[
 B_m(x,R+1) \smallsetminus B_m(x,R) \subset B_{\infty}(x, \eta (R+1)) \smallsetminus B_{\infty}(x,R).
\]
The second part of Lemma \ref{lem_uniform_convergence_of_convex} gives then
\begin{align*}
 \int_{B_m(x,R+1) \smallsetminus B_m(x,R)} \Omega^{F_m} &\leqslant \eta \int_{B_{\infty}(x,\eta(R+1)) \smallsetminus B_{\infty}(x,R)} \Omega^{F_{\infty}}, \\
 \int_{B_m(x,R)} \Omega^{F_m} &\geqslant  \eta^{-1} \int_{B_{\infty}(x,R)} \Omega^{F_{\infty}}.
\end{align*}
Now, since the Hilbert geometry $(S_{\infty},d_{S_{\infty}})$ is isometric to a normed vector space (Proposition \ref{dlh}), it is easy to compute volumes: there is some $C>0$ such that
\begin{align*}
 \int_{B_{\infty}(x,\eta(R+1)) \smallsetminus B_{\infty}(x,R)} \Omega^{F_{\infty}} &= C \voleucl(\S^{n-1})  \left( \eta^n (R+1)^n -R^n \right) \\
 \int_{B_{\infty}(x,R)} \Omega^{F_{\infty}} &= C \voleucl(\S^{n-1}) R^n.
\end{align*}
Finally, we get
$$
  R^{F_{m}}(f_{R,m,x}) \leqslant n \eta^{2} \frac{\eta^n (R+1)^n -R^n}{R^n} = n \eta^{n+2} \left( (1-\eta^{-n}) + \frac{(R+1)^n -R^n}{R^n} \right) \leqslant \eps/2,
$$
where the last inequality is obtained thanks to our assumptions on $R$ and $\eta$.
\end{proof}

We can now finish the
\begin{proof}[Proof of Theorem \ref{thm_small_eigenvalues}]
 Let $\eps>0$ and $N \in \N$. Choose $R >0$ as in Lemma \ref{lem_function_with_small_Rayleigh_quotient} and points $x_1, \dots, x_N$ in $S_{\infty}$ such that the $d_{S_{\infty}}$-distance between each pair of points is at least $2R +3$. Then pick a compact set $K \subset S_{\infty}$ containing all the balls $B_{m}(x_i,R+1)$ for $m$ big enough. Such a compact set exists: for instance, take a compact set which contains the balls $B_{\infty}(x_i,R+3)$; then, for $m$ big enough, we have $B_{m}(x_i,R+1) \subset B_{\infty}(x_i,R+3)$. 

By Lemma \ref{lem_function_with_small_Rayleigh_quotient}, there exists $M \in \N$, such that for $m \geqslant M$, the functions $f_{R,m,x_i}$, $1\leqslant i \leqslant N$, are such that
\begin{equation*}
 R^{F_{m}}(f_{R,m,x_i}) \leqslant \eps/2.
\end{equation*}
Furthermore, the $x_i$ are sufficiently apart so that the functions $f_{R,m,x_i}$ have disjoint support. The Min-Max principle (Theorem \ref{thm_min_max}) allows us to conclude that there are at least $N$ eigenvalues below $\eps$.
\end{proof}

\bibliographystyle{amsplaineprint}
\bibliography{bib_BCCV_article}

\end{document}